\documentclass[a4paper]{article}

\usepackage{a4wide}
\usepackage{amsthm,amsmath,amsfonts,amssymb,bbm,mathrsfs}
\usepackage{authblk}
\usepackage{graphicx}
\usepackage{caption}
\usepackage{cite}
\usepackage{enumitem, upref}
\usepackage{subcaption}
\usepackage{booktabs}
\usepackage[english]{babel}
\usepackage{hyperref} 
\usepackage{multirow,tabularx}
\newtheorem{theorem}{Theorem}

\newtheorem{condition}[theorem]{Condition}

\newtheorem{lemma}[theorem]{Lemma}

\theoremstyle{definition}
\newtheorem{algorithm}{Algorithm}
\newtheorem{example}[theorem]{Example}
\theoremstyle{remark}
\newtheorem{remark}{Remark}

\numberwithin{equation}{section}
\theoremstyle{plain}

\setenumerate{label={\normalfont{(\roman*)}}} 

\newcommand{\sss}[1]{\scriptscriptstyle{#1}}
\newcommand{\inprob}{\overset{\sss{\Prob}}{\longrightarrow}}
\newcommand{\indist}{\overset{d}{\longrightarrow}}
\newcommand{\Prob}{\mathbb{P}}
\newcommand{\Ex}[1]{\mathbb{E}[#1]}
\newcommand\abs[1]{\left|#1\right|}
\newcommand{\prc}{{\sss{(\pi)}}}
\newcommand{\prcn}{{\sss{(\pi_n)}}}
\newcommand{\hie}{{\sss{(\mathrm{H})}}}
\newcommand{\ltd}{\ell^2_{\sss{\downarrow}}}
\newcommand{\me}{\textup{e}}
\newcommand\ind[1]{\mathbbm{1}_{\{#1\}}}

	\title{Mesoscopic scales in hierarchical configuration models}
	\author{Remco van der Hofstad, Johan S.H. van Leeuwaarden, Clara Stegehuis}
	\affil{Eindhoven University of Technology, Department of Mathematics and Computer Science, P.O. Box 513, 5600 MB Eindhoven, The Netherlands}
	
\begin{document}
		\maketitle
	\begin{abstract}
		To understand mesoscopic scaling in networks, we study the hierarchical configuration model (HCM), a random graph model with community structure. The connections between the communities are formed as in a configuration model.
		We study the component sizes of the hierarchical configuration model at criticality when the inter-community degrees have a finite third moment. We find the conditions on the community sizes such that the critical component sizes of the HCM behave similarly as in the configuration model. Furthermore, we study critical bond percolation on the HCM.
		We show that the ordered components of a critical HCM on $N$ vertices are of sizes $O(N^{2/3})$. More specifically, the rescaled component sizes converge to the excursions of a Brownian motion with parabolic drift, as for the scaling limit for the configuration model under a finite third moment condition.
	\end{abstract}

\section{Introduction}
Random graphs serve as basic models for networked structures that occur in many sciences, including physics, chemistry, biology, and the social sciences. In these networked structures, it is common to distinguish between two levels, referred to as ``microscopic'' and ``macroscopic''. Vertex degrees and edges between vertices provide a microscopic description, whereas most network functionalities require a macroscopic picture.  Random graph models are typically defined at the microscopic level, in terms of degree distributions and edge probabilities, leading to a collection of local probabilistic rules. This provides  a mathematical handle to characterize the macroscopic network functionality related to global characteristics such as connectivity, vulnerability and information spreading.

Intermediate or ``mesoscopic'' levels are less commonly considered in random graph models and network theory at large, and apply to substructures between the vertex and network levels. Mesoscopic levels are however becoming increasingly in focus, for example because of community structures or hidden underlying hierarchies, common features of many real-world networks. It is not easy to define what is precisely meant with mesoscopic, apart from the obvious definition of something between microscopic and macroscopic. This paper deals with large-network limits, in which the network size $N$ (number of vertices) will tend to infinity. The mesoscopic scale then naturally refers to structures of size $N^\alpha$, where it remains to be determined what values of $\alpha$ need to be considered.

We will associate the mesoscopic level with the community structure, defined as the collection of subgraphs with dense connections within themselves and sparser ones between them. Once the number and sizes of the network communities are identified, not only the community sizes are mescoscopic characteristics, but also the connectivity between communities and their internal organization. 

To investigate the mesoscopic scales, we shall work with the Hierarchical Configuration Model (HCM)~\cite{hofstad2015}, a random graph model that incorporates community structures. On the macroscopic level, the HCM is a configuration model, but every vertex is then replaced by a community, which is a small connected graph. The
HCM crucially deviates from the classical random graph models,
which are all locally tree-like (contain few short cycles). Many real-world networks are not locally tree-like, and have a community structure, with many short cycles inside the communities. 
The HCM allows to study such networks with community structure. Where the classical configuration model (CM) can create a random graph for any given degree sequence, the HCM can generate random graphs with any given community structure. In~\cite{stegehuis2015, stegehuis2016} the HCM is compared to real-world networks, and found to describe real-world networks much better than the CM, while remaining analytically tractable.
As such, the CM model was found in~\cite{stegehuis2015, stegehuis2016} to be a better model for the mesoscopic scale than for the microscopic scale, where mesoscopic means that each vertex in the CM structure represents a community. The CM is a special case of the HCM with all communities of size one.

To better understand the mesoscopic scale, we shall study the HCM in the critical regime, when the random graphs is on the verge of having a giant connected component. This critical regime has been explored for wide classes of random graph models, including the CM.  Indeed, most random graph models undergo a transition in connectivity, a so-called \textit{phase transition}, as illustrated in Figure~\ref{fig:cmtrans}. The component sizes of random graphs at criticality were first investigated for Erd\H{o}s-R\'enyi random graphs~\cite{aldous1997,bhamidi2014a}, and more recently for inhomogeneous random graphs~\cite{bhamidi2010, turova2009} and for the CM~\cite{dhara2016, joseph2014,nachmias2010, riordan2012}. All these models were found to follow qualitatively similar scaling limits, and hence can be considered to be members of the same universality class.

\begin{figure}
	\centering
	\begin{subfigure}[t]{0.27\textwidth}
		\centering
		\includegraphics[width=\linewidth]{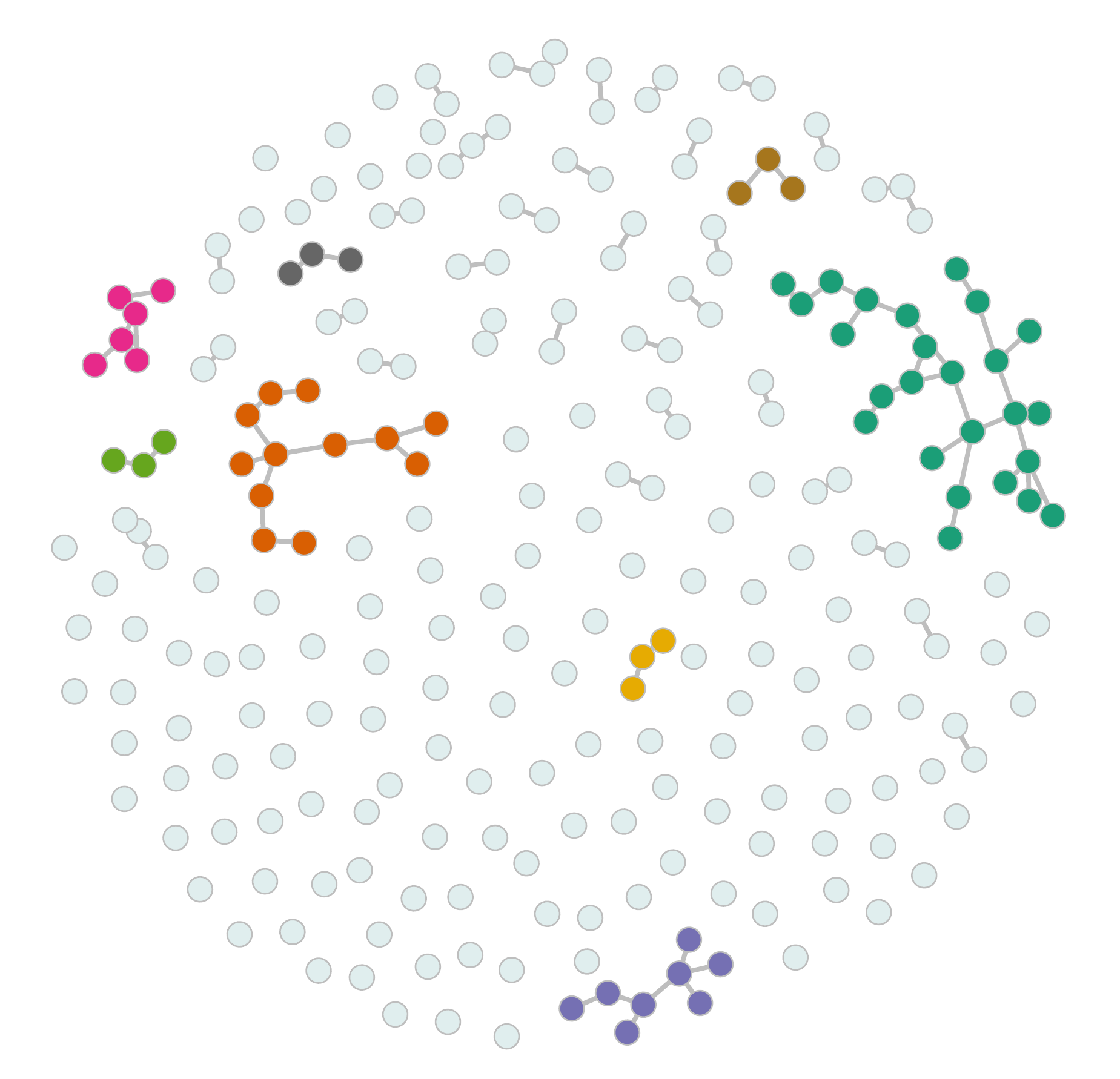}
		\caption{Subcritical} \label{fig:CMsub}
	\end{subfigure}
	\hspace*{\fill} 
	\begin{subfigure}[t]{0.27\textwidth}
		\centering
		\includegraphics[width=\linewidth]{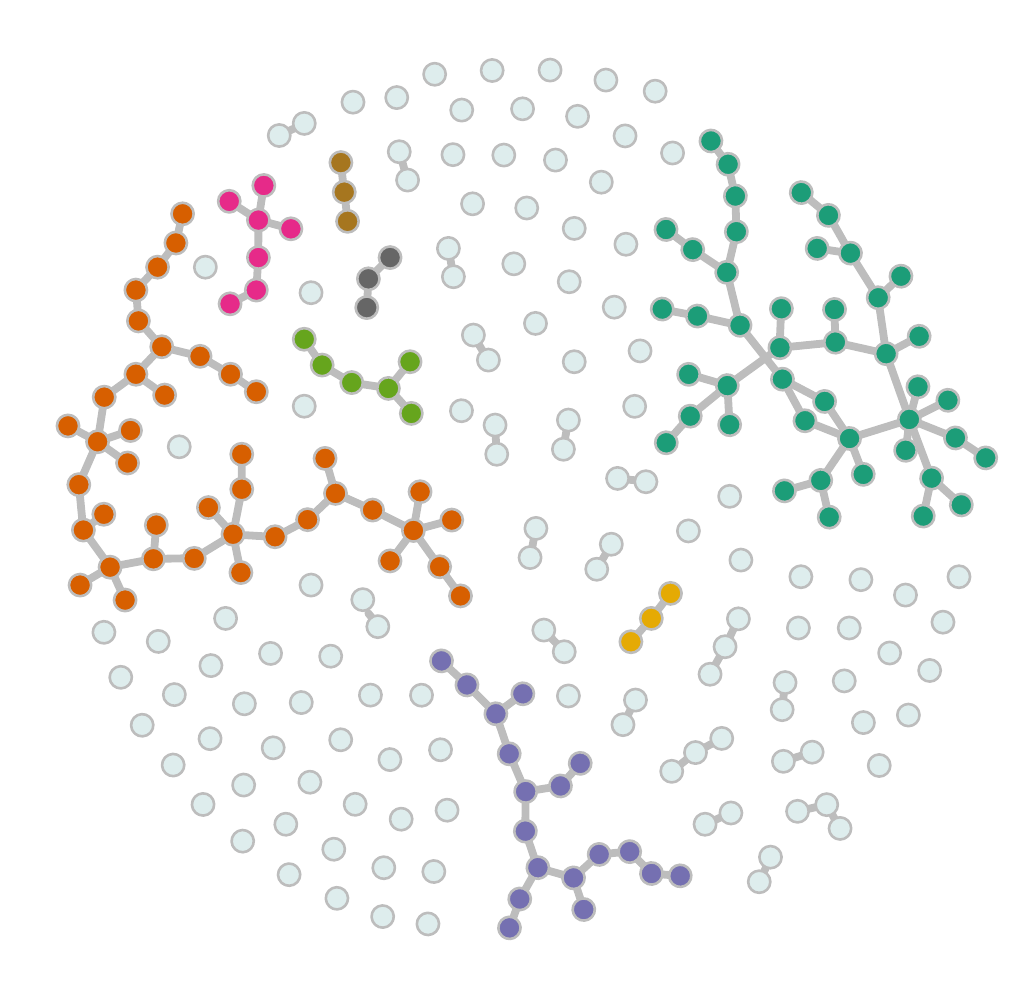}
		\caption{Critical} \label{fig:CMcrit}
	\end{subfigure}
	\hspace*{\fill} 
	\begin{subfigure}[t]{0.27\textwidth}
		\centering
		\includegraphics[width=\linewidth]{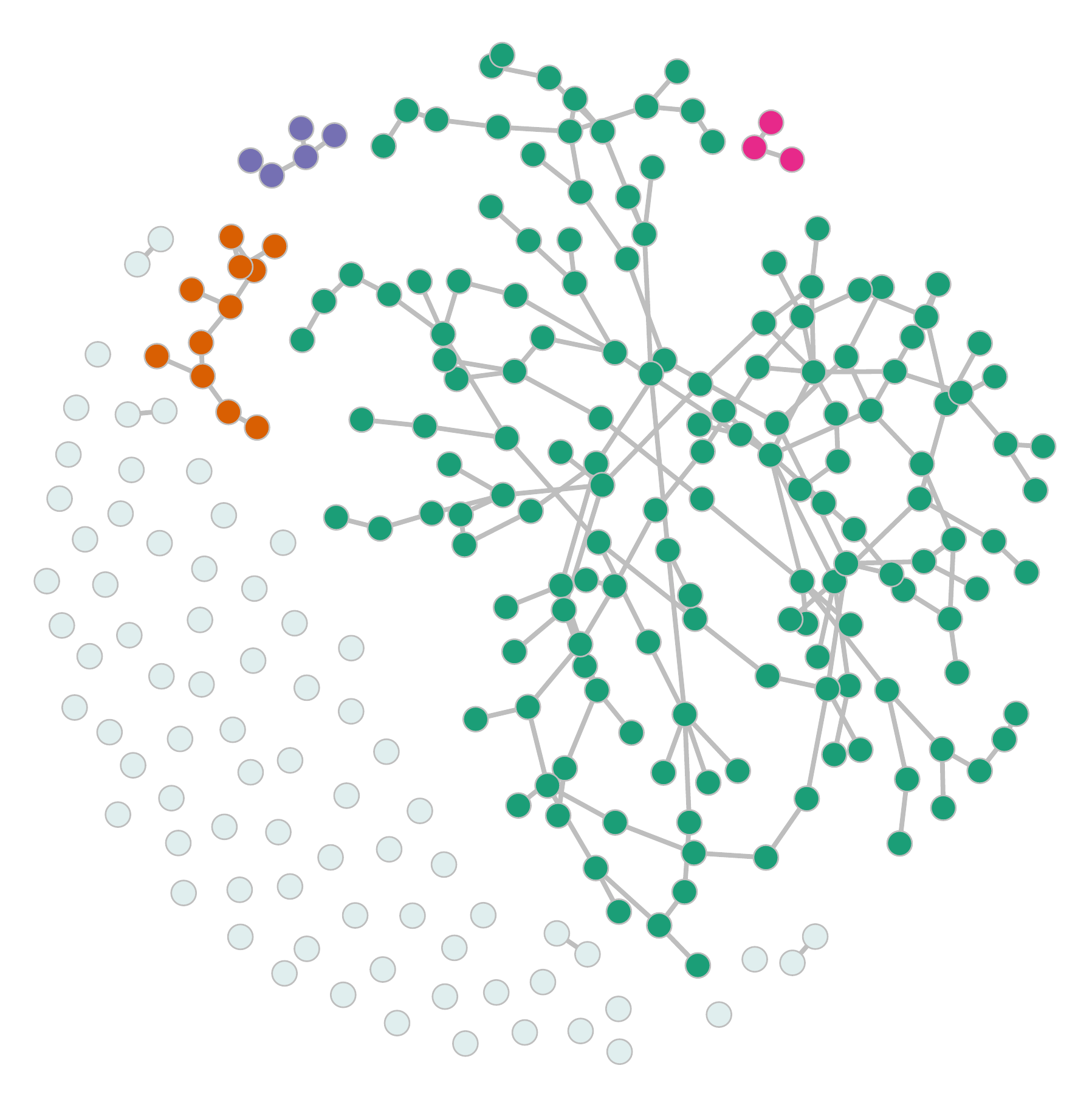}
		\caption{Supercritical} \label{fig:CMsuper}
	\end{subfigure}
	\caption{Phase transition for the component sizes in a configuration model}
	\label{fig:cmtrans}
\end{figure}

Taking the HCM as the null model for studying critical connectivity, we can  investigate the influence of the community structure. A relevant question is under what conditions the HCM will show the same scaling limit as in the classical random graph models and hence is a member of the same universality class.  An alternative formulation of the same question is to ask what the natural order of the mesoscopic scale should be, to influence or even alter the critical graph behavior. Our analysis shows that $\alpha=2/3$ is a strong indicator for the extent to which mesoscopic scales changes the global network picture. When communities are of size $n^{2/3}$ or smaller, the mesoscopic scales are small (yet not negligible), and the critical structures that arise are comparable to the structures encountered in the classical CM. When communities are potentially larger than $n^{2/3}$, the communities themselves start to alter the critical structures, and have the potential to entirely change the macroscopic picture.


We then proceed to study percolation on the HCM. We will exploit the fact that any supercritical HCM can be made critical by choosing a suitable percolation parameter. Therefore, we also study the scaling limits of the component sizes of a HCM under critical bond percolation. We show that under percolation, the community structure not only affects the component sizes, but also the width of the critical window. 

Our main results for the critical components, both before and after percolation, crucially depend on the mesoscopic scale of the community structure. We obtain the precise conditions (Conditions~\ref{cond:graph} and~\ref{cond:size}) under which the mesoscopic scale does not become dominant. These conditions describe the maximum order of the community sizes that can be sustained in order not to distort the picture generated by the CM. In other words, when the community sizes remain relatively small, the results proven for the CM remain valid, despite the fact that the locally tree-like assumption is violated. And equally important, the same conditions indicate when the community sizes become large enough for the mesoscopic scale to take over. In that case, the CM is not an appropriate model.

\paragraph*{Notation.}\label{sec:notation}
We use $\overset{d}\longrightarrow$ for convergence in distribution, and $\inprob $ for convergence in probability. We say that a sequence of events $\mathcal{E}_n$ happens with high probability (w.h.p.) if $\lim_{n\to\infty}\Prob(\mathcal{E}_n)=1$. We write $f(n)=o(g(n))$ if $\lim_{n\to\infty}f(n)/g(n)=0$, and $f(n)=O(g(n))$ if $|f(n)|/g(n)$ is uniformly bounded, where $(g(n))_{n\geq 1}$ is nonnegative. We say that $X_n=O_{\sss{\Prob}}(b_n)$ for a sequence of random variables $(X_n)_{n\geq 1}$ if $|X_n|/b_n$ is a tight sequence of random variables, and $X_n=o_{\sss{\Prob}}(b_n)$ if $X_n/b_n\inprob 0$.

\begin{figure}
	\centering
	\begin{subfigure}[t]{0.3\textwidth}
		\centering
		\includegraphics[width=0.95\linewidth]{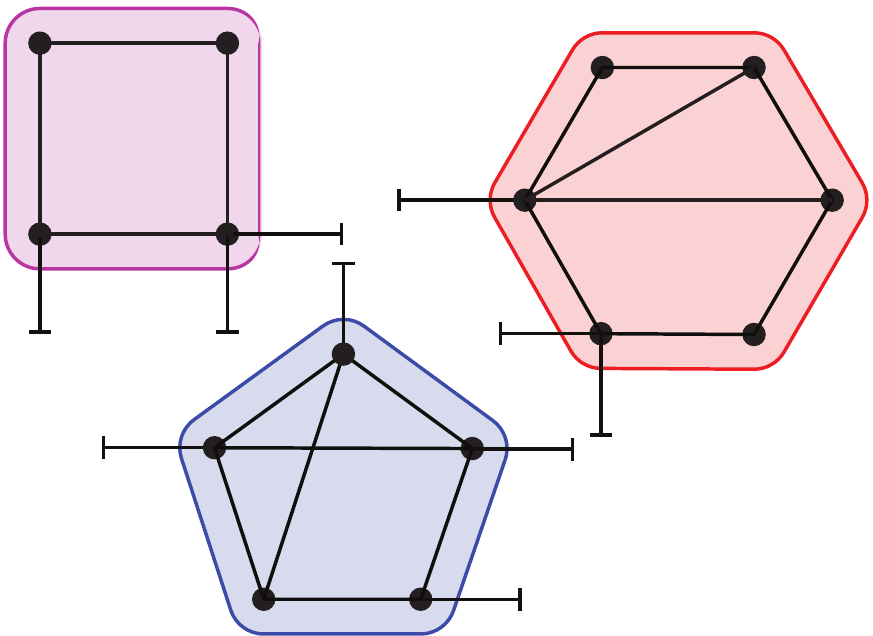}
		\caption{Communities} \label{fig:HCMstubs}
	\end{subfigure}
	\hspace*{0.5cm} 
	\begin{subfigure}[t]{0.3\textwidth}
		\centering
		\includegraphics[width=0.95\linewidth]{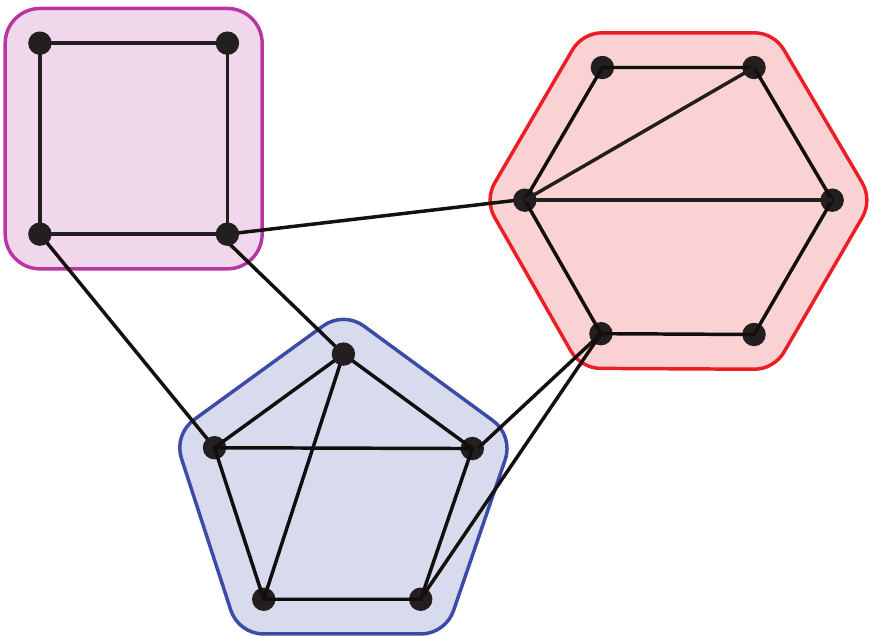}
		\caption{Possible result of HCM} \label{fig:comex}
	\end{subfigure}
	\caption{Illustration of HCM}
	\label{fig:HCM}
\end{figure}

\subsection{Detailed model description}\label{sec:model}

We now describe the HCM in more detail.
Consider a random graph $G$ with $n$ communities $(H_i)_{i\in[n]}$. A community $H$ is represented by $H=(F,\boldsymbol{d})$, where $F=(V_F,E_F)$ is a simple, connected graph and $\boldsymbol{d} = (d_v^{\sss{(b)}})_{v\in V_F}$ , where $d_v^{\sss{(b)}}$ is the number of edges from $v\in V_F$ to other communities. Thus $\boldsymbol{d}$ describes the degrees between the communities. We call $d_v^{\sss{(b)}}$ the \emph{inter-community degree} of a vertex $v$ in community $F$. A vertex inside a community also has an \emph{intra-community degree} $d_v^{\sss{(c)}}$: the number of edges from that vertex to other vertices in the same community. The sum of the intra-community and the inter-community degrees of a vertex is the degree of the vertex, i.e., $d_v=d_v^{\sss{(b)}}+d_v^{\sss{(c)}}$. Let $d_H=\sum_{v\in V_F}d_v^{\sss{(b)}}$ be the total number of edges out of community $H$. Then the (HCM) is formed in the following way. We start with $n$ communities. Every vertex $v$ has $d_b^{\sss{(b)}}$ half-edges attached to it, as shown in Figure~\ref{fig:HCMstubs}. These inter-community half-edges are paired uniformly at random. This results in a random graph $G$ with a community structure, as shown in Figure~\ref{fig:comex}. On the macroscopic level, $G$ is a configuration model with degrees $(d_{H_i})_{i\in[n]}$. We will need to use some assumptions on the parameters of our model. For this, we start by introducing some notation.

Let $H_n$ denote a uniformly chosen community in $[n]=\{1,2,\dots,n\}$. Furthermore, denote the number of communities of type $H$ in a graph with $n$ communities by $n^{\sss{(n)}}_H$. Then $n^{\sss{(n)}}_H/n$ is the fraction of communities that are of type $H$.
Let $D_n$ be the number of outgoing edges from a uniformly chosen community, i.e., $D_n=d_{H_n}=\sum_{v\in V_{F_n}}d_v^{\sss{(b)}}$. Let the size of community $i$ be denoted by $s_i$, and the size of a uniformly chosen community in $[n]$ by $S_n$. Then the total number of vertices in the graph is $N=\sum_{i=1}^ns_i=n\mathbb{E}[S_n]$. Let $S$ and $D$ be the limiting distributions of $S_n$ and $D_n$ respectively, as $n\to\infty$. We assume that the following conditions hold:

\begin{condition}[Community regularity]\label{cond:graph}
	\leavevmode
	\begin{enumerate}
		\item\label{cond:pn}
		$P_n(H)=n^{\sss{(n)}}_H/n\inprob P(H)$, where $P(H)$ is a probability distribution on labeled graphs of arbitrary size.
		\item\label{cond:S}
		$
		\lim_{n\to\infty}\mathbb{E}[S_n]=\mathbb{E}[S]<\infty.$
		\item \label{cond:EDS}
$\mathbb{E}[D_nS_n]\to\mathbb{E}[DS]<\infty$.
\item \label{cond:sd2}
$s_{\text{max}} = \max_{i\in[n]} s_i \ll\frac{n^{2/3}}{\log(n)}$.
	\end{enumerate}
\end{condition}

\begin{condition}[Inter-community connectivity]\label{cond:size}
\leavevmode
\begin{enumerate}
\item \label{cond:ED}
$
\lim_{n\to\infty}\mathbb{E}[D_n^3]=\mathbb{E}[D^3]<\infty.
$
\item \label{cond:PD}
$\Prob(D=0)<1$, $\Prob(D=1)\in(0,1)$.
\item \label{cond:nu}
$\nu_D^{\sss{(n)}}:=\frac{\mathbb{E}[D_n(D_n-1)]}{\mathbb{E}[D_n]}=1+\lambda n^{-1/3}+o(n^{-1/3}),$ for some $\lambda\in\mathbb{R}$.
\end{enumerate}
\end{condition}
\begin{remark}
	Condition~\ref{cond:graph}\ref{cond:pn} implies that $(F_n,\boldsymbol{d}_n)\overset{d}\longrightarrow (F,\boldsymbol{d})$, $D_n\overset{d}{\longrightarrow}D$ and $S_n\overset{d}{\longrightarrow}S$. The condition also implies that $d_{\max}=o(n^{1/3})$~\cite{hofstad2009}, where $d_{\max}$ is the maximal inter-community degree. Define
\begin{align}
p_{k,s}^{\sss{(n)}}&=\sum_{H=(F,\boldsymbol{d}):|F|=s,d_H=k}P_n(H),\\
p_{k,s}&=\sum_{H=(F,\boldsymbol{d}):|F|=s,d_H=k}P(H),
\end{align}
as the probabilities that a uniformly chosen community has size $s$ and inter-community degree $k$. Then Condition~\ref{cond:graph} implies that $p_{k,s}^{\sss{(n)}}\to p_{k,s}$ for every $(k,s)$. Furthermore, let $\ell_n$ denote the sum of the inter-community degrees of all half-edges, $\ell_n=\sum_{i\in[n]}d_{H_i}$.
\end{remark}

\subsection{Results on critical component sizes}\label{sec:critsize}

For a connected component of $G$, we can either count the number of communities in the component, or the number of vertices in it. We denote the number of communities in a component $\mathscr{C}$ by $v^\hie(\mathscr{C})$, and the number of communities with inter-community degree $k$ by $v_k^\hie(\mathscr{C})$. The number of vertices in component $\mathscr{C}$ is denoted by $v(\mathscr{C})$.
Define
\begin{equation}
\nu_D=\frac{\mathbb{E}[D(D-1)]}{\mathbb{E}[D]},
\end{equation}
where $D$ is the asymptotic community degree in Condition~\ref{cond:size}.
Let $p_k=\Prob(D=k)$.

Let $B^\mu_{\lambda,\eta}(t)$ denote Brownian motion with a parabolic drift~\cite{groeneboom1989, hofstad2010}: $B^\mu_{\lambda,\eta}(t)=\frac{\sqrt{\eta}}{\mu}B(t)+\lambda t-\frac{\eta t^2}{2\mu^3}$, where $B(t)$ is a standard Brownian motion.
Let $W^\lambda(t)$ be the reflected process of $B^\mu_{\lambda,\eta}(t)$, i.e.,
\begin{equation}\label{eq:W}
W^\lambda(t)=B^\mu_{\lambda,\eta}(t)-\min_{0\leq s\leq t}B^\mu_{\lambda,\eta}(s),
\end{equation}
and let $\boldsymbol{\gamma}^\lambda$ denote the vector of ordered excursion lengths of $W^\lambda$.
Choose $\mu=\mathbb{E}[D]$, $\eta=\mathbb{E}[D^3]\mathbb{E}[D]-\mathbb{E}[D^2]$.
Let $\mathscr{C}_{(j)}$ denote the $j$th largest component of a HCM, and $\mathscr{C}_{(j)}^{\sss{(\mathrm{CM})}}$ the $j$th largest component of the underlying CM (i.e.,$\mathscr{C}_{(j)}^{\sss{(\mathrm{CM})}}$ is the $j$-th largest component measured in terms of the number of communities).
Since the underlying CM satisfies Condition~\ref{cond:graph}, by~\cite{dhara2016},
\begin{equation}\label{eq:critconf}
n^{-2/3}\left(v^\hie(\mathscr{C}_{(j)}^{\sss{(\mathrm{CM})}})\right)_{j\geq 1}\to \boldsymbol{\gamma}^{\lambda}.
\end{equation}
Thus, the number of communities in the components of a HCM follows the same scaling limit as the configuration model, since the communities are connected as in a configuration model. 

The following theorem shows that the scaled component sizes of a HCM converge to a constant times $\boldsymbol{\gamma}^\lambda$ as well:
\begin{theorem}\label{thm:crit}
	Fix $\lambda\in \mathbb{R}$. For a hierarchical configuration model satisfying \textup{Conditions~\ref{cond:graph}} and~\ref{cond:size},
	\begin{equation}\label{eq:critsize}
	N^{-2/3}(v(\mathscr{C}_{(j)}))_{j\geq 1}\indist \Ex{S}^{-2/3}\frac{\Ex{DS}}{\Ex{S}} \boldsymbol{\gamma}^{\lambda},
	\end{equation}
	with respect to the product topology.
\end{theorem}

In the  Erd\H{o}s-R\'enyi random graph, the inhomogeneous random graph as well as in the CM, the scaled critical component sizes converge in the $\ltd$ topology~\cite{aldous1997, bhamidi2010, dhara2016}, defined as
\begin{equation}
\ltd:=\{\boldsymbol{x}=(x_1,x_2,x_3,\dots): x_1\geq x_2\geq x_3\geq\dots \text{ and } \sum_{i=1}^\infty x_i^2<\infty\},
\end{equation}
with the 2-norm as metric. Theorem~\ref{thm:crit} only proves convergence of the scaled component sizes of a HCM in the product topology. In the CM, the conditions for convergence in the product topology and convergence in the $\ltd$-topology are the same. In the HCM however, the conditions for convergence in the product topology turn out to be different than the conditions for convergence in the $\ltd$-topology:
\begin{theorem}\label{thm:ltd}
	Suppose $G$ is a hierarchical configuration model satisfying \textup{Conditions}~\ref{cond:graph} and~\ref{cond:size}. Then the convergence of Theorem~\ref{thm:crit} also holds with respect to the $\ltd$-topology if and only if $G$ satisfies $\Ex{S_n^2}=o(n^{1/3})$.
\end{theorem}
\begin{remark}
	This theorem shows that there exist graphs where the critical component sizes converge in the product topology, but not in the $\ltd$-topology. As mentioned before, this does not happen in other random graph models such as  Erd\H{o}s-R\'enyi random graph, the inhomogeneous random graph and the CM~\cite{aldous1997, bhamidi2010, dhara2016}. Furthermore, we can find the exact condition under which the component sizes only converge in the product topology.  
	This theorem also shows that the conditions for convergence in the product topology and the $\ltd$-topology are equivalent in the CM: the CM is a special case of the HCM with size one communities. Therefore, $\Ex{S_n}=1$ for the CM, and the component sizes always converge in the $\ltd$-topology if they converge in the product topology.
	It is surprising that the condition on the $\ltd$ convergence only depends on the community sizes, but not on the joint distribution of the community sizes and the inter-community degrees.	
\end{remark}
\begin{remark}
	 The results that we show in this paper can also be applied to a CM or a HCM with vertex attributes. In this setting, every vertex of a CM has a positive, real-valued vertex attribute $S$, where $S$ satisfies Conditions~\ref{cond:graph} and~\ref{cond:size}. These vertex attributes may for example denote the capacity or the weight of a vertex. If we are interested in the sum of the vertex attributes in each connected component, then these sums scale as in Theorem~\ref{thm:crit}. In an even more general setting, the random graph has a community structure, and every vertex within the communities again has a vertex attribute. Then we are back in the HCM setting, only now the community size $S$ does not denote the number of vertices in a community, but the sum over the vertex attributes in a certain community. Therefore, also in a random graph with community structure and vertex attributes, the critical sum over vertex attributes satisfies Theorem~\ref{thm:crit} under appropriate conditions as in Condition~\ref{cond:graph}.
\end{remark}

\subsection{Results on critical percolation on the HCM}\label{sec:percresults}
We now consider bond percolation on the hierarchical configuration model, where every edge is removed independently with probability $1-\pi$.
In the CM, the random graph that remains after percolation can be described in terms of a CM with a different degree sequence~\cite{fountoulakis2007,janson2009b}. Furthermore, if the CM is supercritical, it is possible to choose $\pi$ such that the resulting graph is distributed as a critical CM.
Similarly, after percolating a HCM, we again obtain a HCM, but with a different community distribution since the communities are percolated~\cite{hofstad2015}. By adjusting the parameter $\pi$ we can make sure that the HCM is critical after percolation. Thus, given any supercritical HCM, it is possible to create a critical HCM, by setting $\pi$ correctly.

In the hierarchical configuration model, it is convenient to percolate first only the edges inside communities. This percolation results in a HCM with percolated communities. These percolated communities may be disconnected. However, if we define the connected components of the percolated communities as new communities, we have a new HCM. Let $S_n^\prc$ and $D_n^\prc$ denote the size and degree of communities after percolation only inside the communities with probability $\pi$, and $S^\prc$ and $D^\prc$ their infinite size limits. After this, we percolate only the inter-community edges. This percolation is similar to percolation on the CM, since the inter-community edges are paired as in the CM.

We assume the following:
\begin{condition}[Critical percolation window]\label{cond:perc}
	\leavevmode
	\begin{enumerate}
		\item
		$S_n$ and $D_n$ satisfy \textup{Conditions}~\ref{cond:graph} and~\ref{cond:size}\ref{cond:ED} and~\ref{cond:PD}, and
		\begin{equation}\label{eq:condin}
		\nu_{D^\prcn}^{\sss{(n)}}:=\frac{\mathbb{E}[D_n^\prcn(D_n^\prcn-1)]}{\mathbb{E}[D^\prcn_n]}\to\frac{\mathbb{E}[D^\prc(D^\prc-1)]}{\mathbb{E}[D^\prc]}>1.
		\end{equation}
		\item
		For some $\lambda\in\mathbb{R}$,
		\begin{equation}\label{eq:pic}
		\pi_n = \pi_n(\lambda):=\frac{1}{\nu_{D^{\sss{(\pi_n(\lambda))}}}^{\sss{(n)}}}\left(1+\frac{\lambda}{n^{1/3}}\right).
		\end{equation}
		\end{enumerate}
		\end{condition}
	Here $\pi$ is the solution to $\pi=1/\nu_{D^\prc}$. 
	\begin{remark}
		It can be shown that $\nu_{D^\prc}^{\sss{(n)}}$ is increasing in $\pi$. Thus,~\eqref{eq:pic} has a \textit{unique} solution for every $\lambda\in\mathbb{R}$ when $n$ is large enough.  
		\end{remark}
		Equation~\eqref{eq:condin} makes sure that after percolation of the intra-community edges, the new HCM is supercritical, otherwise there is no hope of making the graph critical by removing more edges. After percolating inside communities, $\nu_{D^{\sss{\pi}}}^{\sss{(n)}}$ is the value of $\nu$ of the new macroscopic CM.
		
		Let $\tilde{D}^\prc$ denote the exploded version of $D^\prc$, that is, every half-edge of a percolated community is kept with probability $\sqrt{\pi}$, and with probability $1-\sqrt{\pi}$, it explodes, it creates a new community of the same shape with only one half-edge attached to it. Then, by~\cite[Thm. 3]{dhara2016}, the component sizes of a percolated CM have similar scaling limits as the original configuration model, but with $D$ replaced by its exploded version. For the HCM, a similar statement holds. Let $\tilde{\boldsymbol{\gamma}}^\lambda$ denote the ordered excursions of the reflected Brownian motion $B_{\eta,\mu}^\lambda$, with $\mu=\Ex{\tilde{D}^\prc}$ and $\eta=\Ex{(\tilde{D}^{\prc })^3}\Ex{\tilde{D}^\prc}-\Ex{(\tilde{D}^{\prc})^2}^2$.
			\begin{theorem}\label{thm:prc}
				Under \textup{Condition}~\ref{cond:perc} in the percolated hierarchical configuration model,
				\begin{equation}\label{eq:sizeperc}
				N^{-2/3}(v(\mathscr{C}_{(j)}))_{j\geq 1}
				\to \mathbb{E}[\tilde{S}^\prc]^{-2/3}\frac{\mathbb{E}[\tilde{D}^\prc\tilde{S}^\prc]}{\mathbb{E}[\tilde{D}^\prc]}\sqrt{\pi}\tilde{\boldsymbol{\gamma}}^\lambda,
				\end{equation}
				in the product topology.
				\end{theorem}
				\begin{remark}
					By~\cite{dhara2016}, a critical CM has
					\begin{equation}\label{eq:picm}
					\pi_n(\lambda)=\frac{1}{\nu^{\sss{(n)}}}(1+\lambda/n^{1/3}) =\pi_n(0)(1+\lambda/n^{1/3}) .
					\end{equation}
					Therefore, the critical window~\eqref{eq:pic} in the HCM is similar to the critical window in the CM, with the difference that in the HCM, we first perform an extra step of percolation inside the communities. For this reason, the critical window of the HCM~\eqref{eq:pic} is in an implicit form, as it depends on both the percolation inside communities, captured in $\nu_{D^{\sss{(\pi_n)}}}^{\sss{(n)}}$, and it depends on the inter-community percolation. In Section~\ref{sec:exppi}, we show that the critical window in the HCM can be written as  
					\begin{equation}\label{eq:pinc}
					\pi_n(\lambda)=\pi_n(0)\left(1+\frac{c^*\lambda}{n^{1/3}}\right),
					\end{equation}
					for some constant $c^*\leq 1$ when $\Ex{D_n^2S_n}\to\Ex{D^2S}<\infty$. In this case, the critical window for the HCM is very similar to the critical window in the CM. 
					The constant $c^*$ captures how much smaller the critical window becomes when adding a community structure. The more vulnerable the communities are to percolation, the smaller the constant $c^*$.
				\end{remark}
				
\begin{remark}
	When percolating inside the communities with parameter $\pi_{\sss{\text{in}}}$ results in a supercritical graph, it is always possible to find $\pi_{\sss{\text{out}}}=1/\nu_{D^{(\pi_{\sss{\text{in}}})}}(1+\lambda/n^{1/3})$ such that the resulting graph is in the critical window with parameter $\lambda$ if we then percolate the inter-community edges with probability $\pi_{\sss{out}}(\lambda)$. The critical value of the HCM is then defined as the value such that $\pi_{\sss{\text{in}}}(\lambda)=\pi_{\sss{\text{out}}}(\lambda)$. Figure~\ref{fig:pinout} illustrates this for several values of $\lambda$ for star-shaped communities.
	
	\begin{figure}
		\centering
		\includegraphics[width=0.4\textwidth]{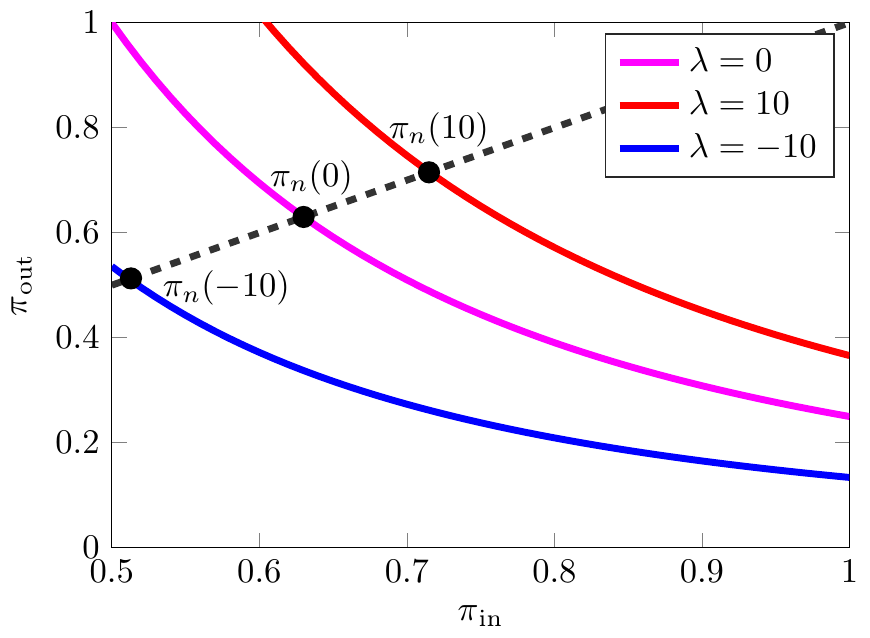}
		\caption{The value of $\pi_{\text{out}}$ corresponding to a given $\pi_{\text{in}}$ for star-shaped communities with five end points (as in Figure~\ref{fig:star}), with $n=10^5$. The intersection with the line $y=x$ gives the critical value $\pi_n(\lambda)$.}
		\label{fig:pinout}
	\end{figure}
	
\end{remark}
\begin{remark}\label{rem:perc}
	Theorem~\ref{thm:prc} shows the convergence of the percolated component sizes in the product topology. As in Theorem~\ref{thm:ltd}, by assuming that $\Ex{S_n^2}=o(n^{1/3})$, we can also show convergence in the $\ltd$-topology. However, in the case of percolation, $\Ex{S_n^2}=o(n^{1/3})$ is not a necessary condition for convergence in the $\ltd$-topology. After percolating first only edges inside communities, we apply Theorem~\ref{thm:crit} or~\ref{thm:ltd} to show the convergence of the percolated clusters.
	In the example of line communities, communities that have the shape of a line, we can make the lines large enough such that $\Ex{S_n^2}>\varepsilon n^{1/3}$. However, if we then percolate inside these line communities, $\Ex{(S_n^\prc)^2}=o(n^{1/3})$. Thus, after percolating inside the communities we can use Theorem~\ref{thm:ltd} to show that the percolated component sizes converge in the $\ltd$-topology, even though $\Ex{S_n^2}\geq \varepsilon n^{1/3}$.
\end{remark}

\subsection{Discussion}
\paragraph{What makes communities ``small"?}
Communities form a mesoscopic structure of a graph. If the communities become very large, then the critical component sizes will be determined by the sizes of the largest communities. In this paper, we find the conditions under which the influence of the mesoscopic structure on the critical component sizes is small when the inter-community degrees have a finite third moment. In this situation, the mesoscopic structures of the configuration model are small enough for the model to be in the same universality class as the configuration model.
Theorems~\ref{thm:crit} and~\ref{thm:ltd} show that there are different scales on which communities can be ``small". Theorem~\ref{thm:crit} shows that the communities are small on the mesoscopic scale when Condition~\ref{cond:graph} holds. In particular, the maximum size of a ``small" community is $n^{2/3}/\log(n)$. This is much smaller than the total number of vertices in the graph, but it still goes to infinity when $n\to\infty$, which shows the mesoscopic nature of the communities. Under these conditions, the order of the components is determined by the order of the components in the underlying CM. Since the convergence of Theorem~\ref{thm:crit} holds in the product topology, this only holds for the first $k$ components, for $k$ fixed.
Theorem~\ref{thm:ltd} shows that an additional condition is necessary for the communities to be small on a macroscopic scale. If this condition does not hold, the first $k$ components are still determined by the macroscopic CM for any fixed $k$, but some large components that are small in the macroscopic sense will be discovered eventually.

\paragraph*{Similarity to the configuration model.}
The scaling limit of Theorem~\ref{thm:crit} is similar to the scaling limit of the configuration model; it only differs by a constant. In fact, the CM is a special case of the HCM: the case where all communities have size one. One could argue that therefore the HCM is in the same universality class as the CM. However, there are still some differences. Note that the variable $D$ in the HCM counts the number of edges going out of the communities. Then, if $D$ has finite third moment, the scaling limit of a HCM is similar to the scaling limit of a CM with finite third moment of the degrees. However, it is possible to construct a HCM with finite third moment of $D$, but infinite third moment of the degree distribution~\cite{hofstad2015}. One example of this is a hierarchical configuration model where all communities are households~\cite{ball2009}: complete graphs, where all vertices of the complete graphs have inter-community degree one. In this household model, every community of  inter-community degree $k$, contains also $k$ vertices of degree $k$. Therefore, the inter-community degree distribution may have finite third moment, while the degree distribution has an infinite third moment. In the CM, the scaling limit under an infinite third moment is very different from the one with finite third moment~\cite{dhara2016a}. However, using this household model, it is possible to construct a random graph with an infinite third moment of the degree distribution, but a similar scaling limit as the CM under the finite third moment assumption.
Similarly, it is possible to create a community structure such that the inter-community degrees have an infinite third moment, but the degree distribution has a finite third moment.
Therefore, adding a community structure to a graph while keeping the degree distribution fixed may change the scaling limits significantly.

\paragraph*{Surplus edges.}
The number of surplus edges of a connected graph $G$ is defined as $\mathsf{SP}(G):=(\# \text{ edges of $G$})-\abs{G}+1$ and indicates how far $G$ deviates from being a tree.
In the CM, the rescaled component sizes and the number surplus edges in the components converge jointly. Note that a surplus edge in the macroscopic CM stays a surplus edge of the HCM, since all communities are connected. In the intuitive picture of densely connected communities, the communities have many surplus edges.
	In the HCM, we give each vertex in the macroscopic CM a weight: the size of the corresponding community. Then, in Theorem~\ref{thm:crit} and~\ref{thm:ltd}, we are interested in the weighted size of the components. Counting surplus edges is very similar, now we also give each vertex in the macroscopic CM a weight: the number of surplus edges in the corresponding community. We are again interested in the weighted component sizes, which counts the total number of surplus edges inside communities. The surplus edges between different communities are the surplus edges of the macroscopic CM. The number of such edges rescaled by $N^{2/3}$ goes to zero by~\cite{dhara2016}. Therefore, if the surplus edges satisfy the same conditions as the community sizes in Condition~\ref{cond:size}\ref{cond:S}, they scale similarly to the component sizes. Thus, if $\mathsf{SP}_n$ denotes the number of surplus edges inside a uniformly chosen community, and $\Ex{\mathsf{SP}_n}\to\Ex{\mathsf{SP}}<\infty$, $\Ex{\mathsf{SP}_n\cdot D_n}\to\Ex{\mathsf{SP}\cdot D}<\infty$ and $\mathsf{SP}_{\max}\ll \frac{n^{2/3}}{\log (n)}$, then
	\begin{equation}\label{eq:surp}
	N^{-2/3}(v(\mathscr{C}_{(j)}), \mathsf{SP}(\mathscr{C}_{(j)}))_{j\geq 1}\indist \Ex{S}^{-2/3} \boldsymbol{\gamma}^{\lambda}\left(\frac{\Ex{DS}}{\Ex{S}},\frac{\Ex{\mathsf{SP}\cdot D}}{\Ex{S}}\right),
	\end{equation}
	in the product topology, where $\mathsf{SP}(\mathscr{C})$ denotes the number of surplus edges in component $\mathscr{C}$.
		
	Unsurprisingly, this scaling of the surplus edges is very different from the scaling of the surplus edges in the CM, which is locally tree-like, even though the scaling of the component sizes is very similar to the one in the CM.
	
	Let $\boldsymbol{N}^\lambda = (N^\lambda(s))_{s\geq 0}$ denote a counting process of marks with intensity $W^\lambda(s)/\Ex{D}$ conditional on $(W^\lambda(u))_{u\leq s}$, with $W^\lambda(s)$ as in~\eqref{eq:W}. Furthermore, let $N(\gamma)$ denote the number of marks in the interval $\gamma$. Let $\mathsf{SP}^\hie(\mathscr{C})$ denote the number of surplus edges of component $\mathscr{C}$ that are inter-community edges. These the surplus edges are the surplus edges of the macroscopic CM. By~\cite{dhara2016}, these surplus edges converge jointly with the component sizes to $N(\boldsymbol{\gamma}^\lambda)$. Therefore, in the HCM,
		\begin{equation}\label{eq:surpconf}
		(N^{-2/3}v(\mathscr{C}_{(j)}), \mathsf{SP}^\hie(\mathscr{C}_{(j)}))_{j\geq 1}\indist \left(\Ex{S}^{-2/3} \boldsymbol{\gamma}^{\lambda}\frac{\Ex{DS}}{\Ex{S}},N(\boldsymbol{\gamma}^\lambda)\right).
		\end{equation}
	
	\paragraph*{Infinite third moment.}
We investigate the scaled component sizes of a random graph with communities, where the inter-community degrees have a finite third moment. A natural question therefore is what happens if we drop the finite third moment assumption.
The scaled component sizes of random graphs with infinite third moment, but finite second moment have been investigated for several models. If the degrees follow a power-law with exponent $\tau\in(3,4)$, then the component sizes of an inhomogeneous random graph as well as a CM scale as $n^{(\tau-2)/(\tau-1)}$~\cite{dhara2016a, bhamidi2012}. In the HCM this may also be the correct scaling, but clearly then we need to replace Condition~\ref{cond:graph}\ref{cond:sd2} by $s_{\max}=o(n^{(\tau-2)/(\tau-1)})$, since otherwise the largest community will dominate the component sizes. This indicates that the heavier the power-law tail, the smaller the maximal community size can be for the HCM to be in the same universality class as the CM. What exact assumptions on the community size distribution are needed to obtain the same scaling limit as in the CM remains open for further research.

\paragraph{Optimality of conditions.}
Condition~\ref{cond:size} is necessary for the macroscopic CM to have component sizes of size $O(n^{2/3})$. Clearly it is also necessary that the maximum community size is $o(n^{2/3})$, since otherwise the largest community could dominate the component sizes. For example, it would be possible to create communities of size larger than $n^{2/3}$ that have no half-edges. Then, these components are the smallest components of the macroscopic CM, but may be the largest components in the HCM. Condition~\ref{cond:graph}\ref{cond:sd2} has an extra factor $1/\log(n)$, which we need to prove that the component sizes are not dominated by the community sizes. Probably this condition is not optimal, we believe the optimal condition to be $s_{\max}=o(n^{2/3})$. Furthermore, Conditions~\ref{cond:graph}\ref{cond:S} and~\ref{cond:EDS} are necessary for taking the limit in~\eqref{eq:critsize}.

\paragraph*{Outline.}
The remainder of this paper is organized as follows. In Section~\ref{sec:crit} we prove Theorem~\ref{thm:crit}. This proof relies heavily on the fact that the macroscopic CM follows a similar scaling limit~\cite{dhara2016}.
Then we prove Theorem~\ref{thm:ltd} in Section~\ref{sec:ltd}.
In Section~\ref{sec:perc}, we study percolation on the HCM. First we prove Theorem~\ref{thm:prc}, and after that we show in Section~\ref{sec:exppi} that the critical window of a HCM is similar to the critical window of the CM. We conclude Section~\ref{sec:perc} with some examples of communities.

\section{Proof of the scaling of the critical HCM}\label{sec:crit}
In this section, we prove Theorem~\ref{thm:crit} and~\ref{thm:ltd}. We start by describing an exploration process that finds the component sizes in Section~\ref{sec:expl}.  We use this exploration process to show that the components that are found before time $Tn^{2/3}$ for some large $T$ converge to the right scaling limit. After that, we prove in Section~\ref{sec:late} that the probability that a large component is found after that time is small, which completes the proof of Theorem~\ref{thm:crit}. Then we prove Theorem~\ref{thm:ltd} in Section~\ref{sec:ltd}.

\subsection{Exploration process.}\label{sec:expl}
To find the component sizes, we use the same depth-first exploration process for the CM as~\cite[Algorithm 1]{dhara2016}. However, instead of exploring vertices, we now explore communities. This means that we only explore the macroscopic CM. In each step of the exploration process, we discover an entire community, and we explore further using only the inter-community connections. Therefore, the only difference between our exploration process and the standard exploration process for the CM, is that we count the number of vertices in each community that we discover. At each step $k$, an inter-community half-edge can be in the set of active half-edges  $\mathcal{A}_k$, in the set of sleeping half-edges $\mathcal{S}_k$, or none of these two sets. Furthermore, every vertex of the HCM is alive or dead. When a vertex is dead, it is in the set of dead vertices $\mathcal{D}_k$.

\begin{algorithm}\label{alg:explore}
	For $k=0$, all inter-community half-edges are in $\mathcal{S}_0$, and both $\mathcal{D}_0$ and $\mathcal{A}_0$ are empty. While $\mathcal{A}_k\neq \varnothing$ or $\mathcal{S}_k\neq\varnothing$ we set $k=k+1$ and perform the following steps:
	\begin{enumerate}
		\item
		If $\mathcal{A}_k\neq\varnothing$, then take the smallest inter-community half-edge $a$ from $\mathcal{A}_k$.
		\item
		Take the half-edge $b$ that is paired to $a$. By construction of the algorithm, the community $H$ to which $b$ is attached, is not discovered yet. Let ${b_{H}}_1,\dots,{b_{H}}_r$ be the other half-edges attached to community $H$, and let $\mathcal{V}_H$ denote the set of vertices of community $H$. Let $b, b_{H1},\dots,b_{Hr}$ be smaller than all other elements of $\mathcal{A}_k$, and order them as ${b_{H}}_1>{b_{H}}_2>\dots>{b_{H}}_r>b$. Let $\mathcal{C}_k\subset \{{b_{H}}_1,\dots,{b_{H}}_r\}$ denote all edges attached to community $H$ that attach to another half-edge adjacent to $H$. Furthermore, let $\mathcal{B}_k\subset \mathcal{A}_k\cup \{{b_{H}}_1,\dots,{b_{H}}_r\}$ denote the collection of half-edges in $\mathcal{A}_k$ that have been paired to one of the ${b_{H}}_i$'s, including the corresponding half-edges incident to community $H$. Then, set $\mathcal{A}_{k+1}=\mathcal{A}_k\cup \{{b_{H}}_1,\dots,{b_{H}}_r\}\setminus(\mathcal{B}_k\cup\mathcal{C}_k)$,
		$\mathcal{S}_{k+1}=\mathcal{S}_k\setminus \{b,{b_{H}}_1,\dots ,{ b_{H}}_r\}$ and $\mathcal{D}_{k+1}=\mathcal{D}_k\cup \mathcal{V}_H$.
		\item
		If $\mathcal{A}_k=\varnothing$, then we pick a half-edge $a$ from $\mathcal{S}_k$ uniformly at random. Let $H$ be the community to which $a$ is attached, and $a_{H1},\dots, a_{Hr}$ the other half-edges attached to community $H$. Again, order them as ${a_{H}}_1> {a_{H}}_2\dots> {a_{H}}_r>a$. Let $\mathcal{V}_H$ denote the vertices of community $H$.  Declare $H$ to be discovered. Let $\mathcal{C}_k$ again denote the collections of half-edges of $H$ in a community self-loop. Then set $\mathcal{A}_{k+1}=\{a,{a_{H}}_1,\dots,{a_{H}}_r\}$, $\mathcal{S}_{k+1}=\mathcal{S}_k\setminus \{a,{a_{H}}_1,\dots , {a_{H}}_r\}$ and $\mathcal{D}_{k+1}=\mathcal{D}_k\cup \mathcal{V}_H$.
	\end{enumerate}
\end{algorithm}

Algorithm~\ref{alg:explore} discovers one community at each step. When the edges going out of the community create a cycle, double edge or self-loop, the corresponding half-edges are in $\mathcal{B}_k$ or $\mathcal{C}_k$, and they are thrown away. Therefore, at each step, an unexplored community is discovered. Since the communities are found by selecting a half-edge at random, the communities are explored in a size-biased manner with respect to the number of edges going out of the community.
The dead vertices correspond to all vertices inside communities that have already been discovered. We define the additive functional
\begin{equation}
Z_n(k)=\abs{\mathcal{D}_k}
\end{equation}
as the number of vertices that have been discovered up to time $k$.
The exploration process finds an undiscovered community of the HCM at each step, and therefore
\begin{equation}
Z_n(k)=\sum_{i=1}^ks_{(i)},
\end{equation}
where $s_{(i)}$ denotes the size of the $i$th discovered community.

Let $d_{(j)}$ be the inter-community degree of the $j$th explored community. Define $Q_n(k)$ as
\begin{equation}\label{eq:Q}
\begin{aligned}[b]
Q_n(0)&=0,\\
Q_n(i)&=\sum_{j=1}^{i}(d_{(j)}-2-2c_{(j)}),
\end{aligned}
\end{equation}
where $c_{(j)}$ denotes the number of cycles/self-loops or double edges that are found when discovering the $j$th community.
Let $\mathscr{C}_k$ denote the $k$th component that is found by the exploration process, and define
\begin{equation}\label{eq:tau}
\tau_k=\inf\{i:Q_n(i)=-2k\}.
\end{equation}
Then, $\tau_k-\tau_{k-1}$ is the number of communities in component $\mathscr{C}_k$~\cite{dhara2016}, so that
\begin{equation}\label{eq:sizeH}
v^\hie(\mathscr{C}_k)=\tau_{k}-\tau_{k-1}.
\end{equation}
Furthermore, the size of $\mathscr{C}_k$ equals
\begin{equation}\label{eq:sizeZ}
v(\mathscr{C}_i)=Z_n(\tau_{k})-Z_n(\tau_{k-1}).
\end{equation}

By~\cite{dhara2016}, the rescaled process $Q_n(t)$ converges to the reflected version of a Brownian motion with negative parabolic drift. To derive the sizes of the component sizes in the hierarchical configuration model, we now study the convergence of the process $Z_n(k)$:

\begin{lemma}
For any $u\geq 0$,
\begin{equation}
\sup_{t\leq u}\abs{n^{-2/3}Z_n(\lfloor tn^{2/3}\rfloor)- \frac{\mathbb{E}[DS]}{\mathbb{E}[D]}t}\inprob 0.
\end{equation}
\end{lemma}
\begin{proof}
We use~\cite[Proposition 29]{dhara2016}, choosing $\alpha=2/3$ and $f_n(i)=s_i$. This yields
\begin{equation}
\sup_{t\leq u}\Big{|}n^{-2/3}\sum_{i=1}^{\lfloor t n^{2/3}\rfloor}s_{(i)}- \frac{\mathbb{E}[DS]}{\mathbb{E}[D]}t\Big{|}=O_\Prob(n^{-1/3}\sqrt{ u s_{\max}}\vee n^{-1/3}u^2d_{\max}).
\end{equation}
Using that $d_{\max}=o(n^{1/3})$ and $s_{\max}= o(n^{2/3})$ by Condition~\ref{cond:graph}\ref{cond:sd2} gives the result.
\end{proof}

\begin{lemma}\label{lem:joint}
For any $u\geq 0$,
\begin{equation}
\left(n^{-1/3}Q_n(tn^{2/3}),n^{-2/3}Z_n(tn^{2/3})\right)_{t\leq u}\indist \left(W^\lambda(t),\frac{\mathbb{E}[DS]}{\mathbb{E}[S]}t\right)_{t\leq u}
\end{equation}
in the $J_1\times J_1$ topology.
\end{lemma}
\begin{proof}
Since $t\mapsto\frac{\mathbb{E}[DS]}{\mathbb{E}[S]}t$ is deterministic, $\left(n^{-1/3}Q_n(tn^{2/3})\right)\indist \left(B^{\lambda}(t)\right)$ by~\cite[Thm. 8]{dhara2016} and $\left(n^{-2/3}Z_n(tn^{2/3})\right)\inprob \left(\frac{\mathbb{E}[DS]}{\mathbb{E}[S]}t\right)$ in the $J_1$ topology, an analogy of Slutsky's theorem for processes proves the lemma.
\end{proof}

By~\cite{dhara2016}, the excursion lengths of $\bar{Q}_n$ converge to $\boldsymbol{\gamma}^\lambda$, where
\begin{equation}
\bar{Q}_n(t)=n^{-1/3}Q_n(tn^{2/3}).
\end{equation}
Since the excursions of $Q_n$ encode the number of communities in the components, and $Z_n$ encodes the sum of the corresponding community sizes, Lemma~\ref{lem:joint} shows that the components that have been discovered before time $Tn^{2/3}$ satisfy~\eqref{eq:critsize}.

\subsection{Sizes of components that are discovered late and convergence in product topology}\label{sec:late}

By Lemma~\ref{lem:joint}, the component sizes that have been discovered up to time $Tn^{2/3}$ converge to a constant times the excursion lengths of a reflected Brownian motion with parabolic drift. To prove that the ordered components of the HCM converge, we need to show that the probability of encountering a large component after time $Tn^{2/3}$, is small. From~\cite[Lemma 14]{dhara2016}, we know that for every $\eta>0$
\begin{equation}
\lim_{T\to\infty}\limsup_{n\to\infty}\Prob\left(v^\hie({\mathscr{C}^{\sss{\geq T}}_{\text{max}}})>\eta n^{2/3}\right)=0,
\end{equation}
where $\mathscr{C}^{\sss{\geq T}}_{\text{max}}$ is the largest component found after time $Tn^{2/3}$. Therefore, we only need to show that the probability that there exists a hierarchical component smaller than $\eta n^{2/3}$ such that its size is larger than $\delta n^{2/3}$ is small when $\eta\ll\delta$. We prove that the probability that a given component $\mathscr{C}$ satisfies this property is exponentially small in $n$. Therefore, the probability that such a component exists is also small.

We first explore the HCM according to Algorithm~\ref{alg:explore} until the first time after $Tn^{2/3}$ a component has been explored. Then, we remove all components that have been found so far. We denote the resulting graph by $G^{\sss{\geq T}}$. The probability $p_{k,s}^{\sss{\geq T}}(n)$ that a community in $G^{\sss{\geq T}}$ has degree $k$ and size $s$ can be bounded as
\begin{equation}\label{eq:pksgt}
p_{k,s}^{\sss{\geq T}}(n)\leq \frac{n_{k,s}}{n-Tn^{2/3}}=p_{k,s}^{\sss{(n)}}(1+o(n^{-1/3})),
\end{equation}
where $n_{k,s}$ denotes the number of communities of size $s$ and inter-community degree $k$.
Therefore, the expected size of a community in $G^{\sss{\geq T}}$, $\Ex{S^{\sss{\geq T}}}<\infty$, and similarly $\Ex{D^{\sss{\geq T}}}<\infty$ and $\Ex{(DS)^{\sss{\geq T}}}<\infty$.
Now, we start exploring $G^{\sss{\geq T}}$ as in Algorithm~\ref{alg:explore}. We want to show that with high probability $G^{\sss{\geq T}}$ does not contain components larger than $\delta n^{2/3}$. By~\cite{dhara2016}, the CM with high probability does not contain any components of size $\eta n^{2/3}$ that are discovered after time $Tn^{2/3}$. Therefore, with high probability $G^{\sss{\geq T}}$ does not contain components with more than $\eta n^{2/3}$ communities.

We now explore $G^{\sss{\geq T}}$ using Algorithm~\ref{alg:explore} for $\eta n^{2/3}$ steps, and investigate the sum of the community sizes that have been explored until time $\eta n^{2/3}$.
Lemma~\ref{lem:exps} shows that the probability of finding more than $\delta n^{2/3}$ vertices after exploring the first $\eta n^{2/3}$ communities is very small. This is a key step in proving Theorem~\ref{thm:crit}.
\begin{lemma}\label{lem:exps}
	For any $\eta, \delta>0$ satisfying $\delta>2\eta \Ex{DS}/\Ex{D}$,
	\begin{equation}
		\Prob\Bigg(\sum_{i=1}^{\eta n^{2/3}}s_{(i)}> \delta n^{2/3}\Bigg)\leq \me^{ -\zeta n^{2/3}/s_{\max}},
	\end{equation}
	for some $\zeta>0$.
\end{lemma}
 \begin{proof}
 	Let $T_i$ be independent exponential random variables with rate $d_i/\ell_n$. Furthermore, let $M(k)=\#\{j:T_j\leq k \}$. Then,
 	\begin{equation}
 	\Ex{M(k)}=\sum_{i\in [n]}\left(1-\me^{-kd_i/\ell_n}\right)\leq k,
 	\end{equation}
 	using that $1-\me^{-x}\leq x$. Similarly, using that $1-\me^{-x}\geq x-x^2/2$,
 	\begin{equation}\label{eq:exmup}
 	\Ex{M(k)}\geq \sum_{i\in [n]}\left(\frac{kd_i}{\ell_n}-\frac{k^2d_i^2}{\ell_n^2}\right)=k-\frac{k^2\Ex{D_n^2}}{n\Ex{D_n}}.
 	\end{equation}
 	Furthermore, for $n$ large enough, if $k=o(n)$,
 	\begin{equation}\label{eq:chern}
 	\begin{aligned}[b]
	 	\Prob(M(2k)<k)&\leq \Prob\left(M(2k)<\frac32 k\left(1-\frac{k\Ex{D_n}}{n\Ex{D_n^2}}\right)\right)\\
	 	& \leq \Prob\left(M(2k)<\frac34\Ex{M(2k)}\right) =\Prob(\me^{-tM(2k)}>\me^{-\frac 34 t \Ex{M(2k)}})\\
	 	&\leq \me^{\frac 34 t \Ex{M(2k)}}\Ex{\me^{-tM(2k)}},
 	\end{aligned}
 	\end{equation}
 	for any $t>0$, where the last inequality uses the Markov inequality. Let $q_i=1-\me^{-\frac{2kd_i}{\ell_n}}$.
 	Since $M(2k)$ is a sum of independent indicator variables,
 	\begin{equation}
 	\Ex{\me^{-tM(2k)}}=\prod_{i\in[n]}\left(1+q_i(\me^{-t}-1)\right)\leq \prod_{i\in[n]}\me^{q_i(\me^{-t}-1)}=\me^{\Ex{M(2k)}(\me^{-t}-1)},
 	\end{equation}
 	where the inequality uses that $1+x\leq \me^{x}$, with $x=q_i(\me^{-t}-1)$. Plugging this into~\eqref{eq:chern} and setting $t=-\log(\frac34)$ yields
 	\begin{equation}
 	\begin{aligned}[b]
 		\Prob(M(2k)<k)&\leq \me^{\Ex{M(2k)}(\frac 34 t+\me^{-t}-1)}= \me^{\Ex{M(2k)}(-\frac 34 \log(\frac 34)-\frac 14)}.
 	\end{aligned}
 	\end{equation}
 	Then, using that $-(1-x)\log(1-x)\leq x-\frac{x^2}{2}$,
 	\begin{equation}
 	\Prob(M(2k)<k)\leq\me^{-\frac{\Ex{M(2k)}}{32}}\leq \me^{-Ck(1-\frac kn)},
 	\end{equation}
 	for some $C>0$, 
 	where we used~\eqref{eq:exmup}.
 	
 	Consider $Y_k=\sum_{i=1}^{M(k)}s_{(i)}-k\frac{\Ex{D_nS_n}}{\Ex{D_n}}$. Let $\mathscr{F}_k$ denote the sigma-field generated by the information revealed up to time $k$. Then,
 	\begin{equation}
 	\begin{aligned}[b]
	 	\Ex{Y_k\mid \mathscr{F}_{k-1}}&=Y_{k-1}+\sum_{i\in[n]}s_i\Prob(T_i\in[k-1,k])-\frac{\Ex{D_nS_n}}{\Ex{D_n}}\\
	 	&=Y_{k-1}+\sum_{i\in[n]}s_i(1-\me^{-d_i/\ell_n})-\frac{\Ex{D_nS_n}}{\Ex{D_n}}\leq Y_{k-1}.
 	\end{aligned}
 	\end{equation}
 Therefore, $Y_k$ is a supermartingale. Furthermore,
 \begin{equation}
 \begin{aligned}[b]
	 \text{Var}(Y_k\mid \mathscr{F}_{k-1})&=\text{Var}(\sum_{i\in [n]}s_i\ind{T_i\in[k-1,k]})=\sum_{i\in[n]}s_i^2(1-\me^{-d_i/\ell_n})\me^{-d_i/\ell_n}\\
	 &\leq \sum_{i\in[n]}s_i^2d_i/\ell_n\leq s_{\max}\frac{\Ex{D_nS_n}}{\Ex{D_n}}.
 \end{aligned}
 \end{equation}
Thus, we can apply~\cite[Thm. 7.3]{chung2006}, which states that for a supermartingale $X$ with $\text{Var}(X_i\mid \mathscr{F}_{i-1})\leq \sigma^2_i$ and $X_i-\Ex{X_i\mid \mathscr{F}_{i-1}}\leq K$ for all $i$,
\begin{equation}
\Prob(X_n\geq X_0+t)\leq\exp\left(-\frac{t^2}{\sum_{i=1}^{n}\sigma^2_i+Kt/3}\right).
\end{equation}
Applying this to $Y_{\lfloor 2\eta n^{2/3}\rfloor}$ with $Y_0=0$, $\sigma_i^2=s_{\max}$ and $K=s_{\max}$, we obtain
\begin{equation}
\Prob(Y_{\lfloor 2\eta n^{2/3}\rfloor}>t)\leq \exp\left(-\frac{t^2}{2 \eta n^{2/3}s_{\max}+s_{\max}t/3}\right).
\end{equation}
Because by assumption $\xi=\delta-2\eta\Ex{DS}/\Ex{D}>0$,
\begin{equation}
\begin{aligned}[b]
	\Prob\Bigg(\sum_{i=1}^{M(2n^{2/3}\eta)}s_{(i)}>\delta n^{2/3} \Bigg)&=\Prob\left(Y_{\lfloor 2n^{2/3}\eta\rfloor}>\delta n^{2/3}-2\eta n^{2/3}\frac{\Ex{D_nS_n}}{\Ex{D_n}}\right)\\
	&=\Prob\left(Y_{\lfloor 2n^{2/3}\eta\rfloor}>\xi n^{2/3}\right)\\
	&\leq  \exp\Big(-\frac{\xi^2n^{4/3}}{2\eta n^{2/3}s_{\max}+s_{\max}\xi n^{2/3}/3}\Big)\leq \me^{-\zeta n^{2/3}/s_{\max}},
\end{aligned}
\end{equation}
for some $\zeta>0$.
Thus,
\begin{equation}
\begin{aligned}[b]
\Prob\Bigg(\sum_{i=1}^{\eta n^{2/3}}s_{(i)}> \delta n^{2/3}\Bigg)&\leq \Prob\Bigg(\sum_{i=1}^{M(2\eta n^{2/3})}s_{(i)}> \delta n^{2/3}\Bigg)+\Prob(M(2\eta n^{2/3})<\eta n^{2/3})\\
&\leq \me^{-\zeta n^{2/3}/s_{\max}}+\me^{-Cn^{2/3}(1-n^{-1/3})},
\end{aligned}
\end{equation}
which proves the lemma.
\end{proof}

Using the previous lemma, we can now show that the probability that a component of size $\delta n^{2/3}$ is found after time $Tn^{2/3}$ is small for $T$ large enough:
\begin{lemma}\label{lem:latecomp}
	Let $\mathscr{C}_{\max}^{\sss{\geq T}}$ denote the largest component of a hierarchical configuration model satisfying \textup{Conditions}~\ref{cond:graph} and~\ref{cond:size}, of which the first vertex is explored after time $Tn^{2/3}$. Then, for all $\delta>0$,
	\begin{equation}
	\lim_{T\to\infty}\limsup_{n\to\infty}\Prob(v(\mathscr{C}_{\max}^{\sss{\geq T}})\geq \delta n^{2/3})=0.
	\end{equation}
\end{lemma}
\begin{proof}
	We condition on the size of the components of the underlying configuration model. Choose $\eta>0$ satisfying $\delta>2\eta \Ex{DS}/\Ex{D}$.
	Then,
	\begin{equation}\label{eq:cmax}
	\begin{aligned}[b]
	\Prob&(v(\mathscr{C}_{\max}^{\sss{\geq T}})>\delta n^{2/3})\\
	&=\Prob(v(\mathscr{C}_{\max}^{\sss{\geq T}})>\delta n^{2/3} \mid v^\hie({\mathscr{C}^{\sss{\geq T}}_{\text{max}}})\leq \eta n^{2/3})\Prob (v^\hie({\mathscr{C}^{\sss{\geq T}}_{\text{max}}})\leq \eta n^{2/3})\\
	&\quad+\Prob(v(\mathscr{C}_{\max}^{\sss{\geq T}})>\delta n^{2/3} \mid v^\hie({\mathscr{C}^{\sss{\geq T}}_{\text{max}}})> \eta n^{2/3})\Prob (v^\hie({\mathscr{C}^{\sss{\geq T}}_{\text{max}}})> \eta n^{2/3})\\
	&\leq \Prob(v(\mathscr{C}_{\max}^{\sss{\geq T}})>\delta n^{2/3} \mid v^\hie({\mathscr{C}^{\sss{\geq T}}_{\text{max}}})\leq \eta n^{2/3})+\Prob (v^\hie({\mathscr{C}^{\sss{\geq T}}_{\text{max}}})> \eta n^{2/3}).
	\end{aligned}
	\end{equation}
	By~\cite{dhara2016}, for any $\eta>0$,
	\begin{equation}\label{eq:Cmaxcm}
	\lim_{T\to\infty}\limsup_{n\to\infty}\Prob (v^\hie({\mathscr{C}^{\sss{\geq T}}_{\text{max}}})> \eta n^{2/3})=0,
	\end{equation}
	so that the second term in~\eqref{eq:cmax} tends to zero.
	
	Now we study the first term in~\eqref{eq:cmax}.
	Given any component $\mathscr{C}$, we start exploring at a vertex of that component, until time $\eta n^{2/3}$. In Lemma~\ref{lem:exps}, the probability that more than $\delta n^{2/3}$ vertices have been found at time $\eta n^{2/3}$ is quite small. Furthermore, we know that $\mathscr{C}$ has been fully explored, since $v^\hie(\mathscr{C}_{\max}^{\sss{\geq T}})<\eta n^{2/3}$.
	 Then, by the union bound and by Lemma~\ref{lem:exps},
	\begin{equation}
	\begin{aligned}[b]
	\Prob\left(v(\mathscr{C}_{\max}^{\sss{\geq T}}\right)>\delta n^{2/3} &\mid v^\hie({\mathscr{C}^{\sss{\geq T}}_{\text{max}}})\leq \eta n^{2/3})\\
&	\leq \sum_{j=1}^{n}\Prob(v(\mathscr{C}_{j})>\delta n^{2/3} \mid v^\hie({\mathscr{C}^{\sss{\geq T}}_{\text{max}}})\leq \eta n^{2/3})\\
	&= \frac{\sum_{j=1}^{n}\Prob(v(\mathscr{C}_{j})>\delta n^{2/3} , v^\hie({\mathscr{C}^{\sss{\geq T}}_{\text{max}}})\leq \eta n^{2/3})}{\Prob(v^\hie({\mathscr{C}^{\sss{\geq T}}_{\text{max}}})\leq \eta n^{2/3})}\\
		&\leq \frac{\sum_{j=1}^{n}\Prob(v(\mathscr{C}_{j})>\delta n^{2/3} , v^\hie({\mathscr{C}_j^{\sss{\geq T}}})\leq \eta n^{2/3})}{\Prob(v^\hie({\mathscr{C}^{\sss{\geq T}}_{\text{max}}})\leq \eta n^{2/3})}\\
	&\leq \frac{\sum_{j=1}^{n}\Prob(v(\mathscr{C}_{j})>\delta n^{2/3} \mid v^\hie({\mathscr{C}_j^{\sss{\geq T}}})\leq \eta n^{2/3})}{\Prob(v^\hie({\mathscr{C}^{\sss{\geq T}}_{\text{max}}})\leq \eta n^{2/3})}\\
	&\leq \frac{n \mathrm{e}^{-\zeta n^{2/3}/s_{\max}}}{\Prob(v^\hie({\mathscr{C}^{\sss{\geq T}}_{\text{max}}})\leq \eta n^{2/3})},
	\end{aligned}
	\end{equation}
	for some $\zeta>0$.
	Since $s_{\max} \ll n^{2/3}/\log(n)$, using~\eqref{eq:Cmaxcm} and taking limits proves the lemma.
\end{proof}

\begin{proof}[Proof of Theorem~\ref{thm:crit}]
	By~\cite{dhara2016}, the excursions of the process $\bar{Q}_n(t)$ defined in~\eqref{eq:Q} converge to $\boldsymbol{\gamma}^\lambda$, the excursions of a reflected Brownian motion with parabolic drift. Then, by Lemma~\ref{lem:joint} and~\eqref{eq:sizeZ}, the component sizes of the HCM that have been discovered up to time $Tn^{2/3}$ for some $T>0$ converge to $\Ex{DS}/\Ex{D}\boldsymbol{\gamma}^\lambda$. Combining this with Lemma~\ref{lem:latecomp} then shows that
	\begin{equation}
	n^{-2/3}(v(\mathscr{C}_{(j)}))_{j\geq 1}\to \frac{\Ex{DS}}{\Ex{S}} \boldsymbol{\gamma}^{\lambda},
	\end{equation}
	in the product topology. 
	Then, using that $N/n\to \Ex{S}$ completes the proof of Theorem~\ref{thm:crit}.
\end{proof}

\subsection{Convergence in $\ltd$ topology: Proof of \textup{Theorem}~\ref{thm:ltd}.}\label{sec:ltd}
To prove Theorem~\ref{thm:ltd}, we show that the probability that a uniformly chosen vertex is in a large component is small, by using the Markov inequality. Thus, we need to bound the expected component size of a uniformly chosen vertex in a HCM. To do this, we bound the expected component size of a uniformly chosen community of size $s$ and inter-community degree $k$ in Lemma~\ref{lem:Evks}. To prove Lemma~\ref{lem:Evks}, we first count the number of paths in the macroscopic configuration model in Lemma~\ref{lem:path}: the number of paths from community to community, ignoring the internal community structures.
 		Let $P_{L}^{\sss{(H_0)}}$ be the set of all macroscopic paths of length $L$ in a HCM, starting from community $H_0$. Furthermore, define $PW_{L}^{\sss{(H_0)}}$ as the number of macroscopic paths of length $L$, starting in $H_0$, weighted by the size of the last community, i.e.,
 		\begin{equation}
	 		PW_{L}^{\sss{(H_0)}}=\sum_{	P_{L}^{\sss{(H_0)}}}s_{\text{last community}}.
 		\end{equation}
 		
 	\begin{lemma}\label{lem:path}
 		For any $L<\frac14n$ and for some $K>0$,
 		\begin{equation}
 			\Ex{PW_{L}^{\sss{(H_0)}}}\leq K\frac{\Ex{D_nS_n}}{\Ex{D_n}} d_{H_0}\nu_n^{L-1}.
 		\end{equation}
 	\end{lemma}
 	
 	\begin{proof}
 		This proof is very similar to the proof of~\cite[Lemma 5.1]{janson2009}.
 		If $L=1$, then the equation states that $	\Ex{PW_{1}^{\sss{(H_0)}}}\leq K\frac{\Ex{D_nS_n}}{\Ex{D_n}}d_{H_0}$, which is true, since there are at most $d_{H_0}$ paths from $H_0$, and the expected weight of the community at the end of the path is $\Ex{D_nS_n}/\Ex{D_n}$.
 		
 		For $L\geq 2$, the path consists of communities $H_0,H_1,\dots, H_{L}$. This path consists of two half-edges at communities $H_1,\dots H_{L-1}$, and one half-edge at the start and at the end of the path. The probability that these half-edges are paired is $(\ell_n-1)^{-1}(\ell_n-3)^{-1}\dots (\ell_n-2L+1)^{-1}$. Therefore,
 		\begin{equation}
 		\Ex{	PW_{L}^{\sss{(H_0)}}} = \frac{d_{H_0}\sideset{}{^*}\sum_{i_1,\dots , i_L}\prod_{j=1}^{L-1}d_{i_j}(d_{i_j}-1)d_{i_L}s_{i_L}}{\prod_{j=1}^{L}(\ell_n-2j+1)},
 		\end{equation}
 		where $\sum\nolimits^*$ denotes the sum over distinct indices, since all communities in the path must be distinct. If we only sum over $i_L\neq \{H_0,i_1,\dots i_{L-1} \}$, we obtain
 		\begin{equation}
 		\begin{aligned}[b]
	 		\sum_{i_L\neq \{H_0,i_1,\dots i_{L-1} \}}d_{i_L}s_{i_L}&=\sum_{i\in[n]}d_is_i-d_{H_0}s_{H_0}-\sum_{j=1}^{L-1}d_{i_j}s_{i_j}\leq n\Ex{D_nS_n}-2(L-1)-1\\
	 		&\leq \ell_n\frac{\Ex{D_nS_n}}{\Ex{D_n}}-2L-1
											 		\leq \frac{\Ex{D_nS_n}}{\Ex{D_n}}\left(\ell_n-2L-1\right),
	 		\end{aligned}
 		\end{equation}
 		where we used that $d_{i_j}\geq 2$ for $j=1,\dots L-1$ and that $s_i\geq 1$ for all $i$.
 		Therefore,
 		\begin{equation}
 		\begin{aligned}[b]
 		\Ex{	PW_{L}^{\sss{(H_0)}}}
 		&\leq \frac{\Ex{D_nS_n}}{\Ex{D_n}}\frac{d_{H_0}\sideset{}{^*}\sum_{i_1,\dots, i_{L-1}}\prod_{j=1}^{L-1}d_{i_j}(d_{i_j}-1)}{\prod_{j=1}^{L-1}(\ell_n-2j+1)}\\
 		&\leq\frac{\Ex{D_nS_n}}{\Ex{D_n}} (n\Ex{D_n})^{-L+1}\frac{d_{H_0}\sideset{}{^*}\sum_{i_1,\dots, i_{L-1}}\prod_{j=1}^{L-1}d_{i_j}(d_{i_j}-1)}{\prod_{j=1}^{L-1}(1-2j/\ell_n)}.
 		\end{aligned}
 		\end{equation}
 		By arguments of~\cite[Lemma 5.1]{janson2009}
 		\begin{equation}
	 		\sideset{}{^*}\sum_{i_1,\dots, i_{L-1}}\prod_{j=1}^{L-1}d_{i_j}(d_{i_j}-1)\leq \left(n\Ex{D_n(D_n-1)}\right)^{L-1}\prod_{j=0}^{L-2}\left(1-\frac{j}{r}\right),
 		\end{equation}
 		where $r$ denotes the number of communities with inter-community degree larger than or equal to 2. Since $r\leq \frac12\ell_n$,
 		\begin{equation}
 		\begin{aligned}[b]
	 			\Ex{	PW_{L}^{\sss{(H_0)}}} &\leq \frac{\Ex{D_nS_n}}{\Ex{D_n}}\left(\frac{\Ex{D_n(D_n-1)}}{\Ex{D_n}}\right)^{L-1}\frac{d_{H_0}\prod_{j=0}^{L-2}\left(1-\frac{j}{r}\right)}{\prod_{j=1}^{L-1}(1-\frac{2j}{\ell_n})}\\
	 		&	\leq \frac{\Ex{D_nS_n}}{\Ex{D_n}} \nu_n^{L-1}d_{H_0}\frac{\prod_{j=0}^{L-2}\left(1-\frac{2j}{\ell_n}\right)}{\prod_{j=1}^{L-1}(1-\frac{2j}{\ell_n})}\\
	 		&	\leq\frac{\Ex{D_nS_n}}{\Ex{D_n}} \nu_n^{L-1}d_{H_0}\left(1-\frac{2L-2}{n\Ex{D_n}}\right)^{-1}\\
	 		&	\leq \frac{\Ex{D_nS_n}}{\Ex{D_n}} \nu_n^{L-1}d_{H_0}\left(1-\frac{1}{2\Ex{D_n}}\right)^{-1},
	 			\end{aligned}
 		\end{equation}
 		where we have used that $L<\tfrac14 n$. This proves the claim, since $\Ex{D_n}>1$.
 	\end{proof}
 	Using Lemma~\ref{lem:path}, we can bound the expected component size in a HCM. We are interested in the expected component size of a randomly chosen community of size $s$ and inter-community degree $k$, $v(\mathscr{C}_{(k,s)})$.
 	
 	\begin{lemma}\label{lem:Evks}
 		For some $C>0$,
 		\begin{equation}
 		\Ex{v(\mathscr{C}_{(k,s)})}\leq s+ C\frac{k}{1-\nu_n}+o(1).
 		\end{equation}
 	\end{lemma}
 	\begin{proof}
 		We split the expectation into two different parts,
 		\begin{equation}\label{eq:Evkssplit}
 		\Ex{v(\mathscr{C}_{(k,s)})}=\Ex{v(\mathscr{C}_{(k,s)})\ind{v^\hie(\mathscr{C}_{\max})\leq \frac14 n}}
 		+\Ex{v(\mathscr{C}_{(k,s)})\ind{v^\hie(\mathscr{C}_{\max})> \frac14 n}}.
 		\end{equation}
 		We bound the first part similar to the argument in~\cite[Lemma 4.6]{janson2009a}. For every community $H'$  in the same component as community $H_0$, there is at least one path between $H_0$ and $H'$. Furthermore, $H'$ adds $s_{H'}$ vertices to the component. Therefore,
 		\begin{equation}
 		v(\mathscr{C}(H_0))\leq \sum_L PW_L^{\sss{(H_0)}}
 		\end{equation}
 		This yields
 		\begin{equation}
	 		\Ex{v(\mathscr{C}_{(k,s)})}\leq \sum_L \Ex{PW_L^{\sss{H_{(k,s)}}}},
 		\end{equation}
 		where $H_{(k,s)}$ is a community of size $s$ and inter-community degree $k$.
 		The sum of the first term in~\eqref{eq:Evkssplit} only goes up to $L=\tfrac{1}{4}n$, since the maximal path size is smaller than the maximal component size. Thus, by Lemma~\ref{lem:path},
 		\begin{equation}
 		\begin{aligned}[b]
 		\Ex{v(\mathscr{C}_{(k,s)})\ind{v^\hie(\mathscr{C}_{\max})\leq \frac14 n}} &\leq s+\sum_{L=1}^{\frac14 n} \Ex{PW_L^{\sss{H_(k,s)}}}
 		\leq s+\frac{\Ex{D_nS_n}}{\Ex{D_n}}Kk\sum_{L=1}^{\infty}\nu_n^{L-1}\\
 		&=s+\frac{\Ex{D_nS_n}}{\Ex{D_n}}\frac{kK}{1-\nu_n}.
 		\end{aligned}
 		\end{equation}
 		
 		For the second term, we use that the maximal component size is bounded from above by the total number of vertices $N=\Ex{S_n}n$. Then we need to bound the probability that the maximal hierarchical component is at least $\tfrac{1}{4}n$. This is the probability that the size of the largest component in a regular configuration model is larger than $\tfrac{1}{4}n$.
 		We can use the same arguments as in~\cite[Lemma 14]{dhara2016} to show that
 		\begin{equation}\label{eq:Ckssmall}
 		\Prob\left(v^\hie(\mathscr{C}_{\max})>\frac 14 n\right)\leq \frac{16\Ex{D_n}}{n(1-\nu_n)}+o(n^{-1}).
 		\end{equation}
 		This gives
 		\begin{equation}\label{eq:Ckslarge}
	 		\Ex{v(\mathscr{C}_{(k,s)})\ind{v^\hie(\mathscr{C}_{\max})> \frac14 n}}\leq  \frac{16N\Ex{D_n}}{n(1-\nu_n)}+o(1)\leq \frac{16k\Ex{S_n}\Ex{D_n}}{(1-\nu_n)}+o(1),
 		\end{equation}
 		for $k>0$.
		Combining~\eqref{eq:Ckssmall} and\eqref{eq:Ckslarge} then yields
		\begin{equation}
		\Ex{v(\mathscr{C}_{(k,s)})}\leq s+ C\frac{k}{1-\nu_n}+o(1),
		\end{equation}
		for some $C>0$.
 	\end{proof}

  \begin{proof}[Proof of Theorem~\ref{thm:ltd}]
%
%
%
  	
  We first show that $\Ex{S_n^2}=o(n^{1/3})$ is a necessary condition for convergence in the $\ltd$-topology.
  	  Thus, we assume that $\Ex{S_n}\geq \varepsilon n^{1/3}$ for some $\varepsilon>0$. Let $i_T$ denote the index of the first component that is explored after time $Tn^{2/3}$. Then, for any $\delta>0$,
  	    	 \begin{equation}
  	    	 \Prob(\sum_{i\geq i_T} v(\mathscr{C}_{(i)}^{\sss{\geq T}})^2 \geq \delta n^{4/3})
  	    	 \end{equation}
  	    	 needs to be small  for convergence in the $\ltd$ topology. Because $x^2$ is a convex function,
  	  \begin{equation}
  	  \sum_{i\geq i_T} v(\mathscr{C}_{(i)})^2  \geq \sum_{j\geq Tn^{2/3}}s_{(j)}^2=\sum_{s}n^{\sss{\geq T}}_{s}s^2 = \sum_s n_s s^2-\sum_s n_s^{\sss{\leq T}}s^2,
  	  \end{equation}
  	  where $n^{\sss{\geq T}}_s$ and $n_s^{\sss{\leq T}}s^2$ denote the number of communities of size $s$, discovered after or before time $Tn^{2/3}$ respectively.
  	  
  	  We can use a martingale argument similar to~\cite[Proposition 29]{dhara2016}, to show that 
  	  \begin{equation}
  	    \sup_{u\leq t}\abs{n^{-2/3}\sum_{i=1}^{\lfloor u n^{2/3}\rfloor}s_{(i)}^2-\frac{\sum_{k,s} ks^2 n_{k,s}}{\ell_n}u} = o_\Prob(n^{2/3}).
  	  \end{equation}
  	  Therefore
  	  \begin{equation}
  	  \sum_s n_s^{\sss{\leq T}}s^2 =Tn^{2/3} \frac{\sum_{k,s} ks^2 n_{k,s}}{\ell_n}+o(n^{4/3})= T\frac{o(n)}{\ell_n} \sum_s s^2 n_{s}+o_\Prob(n^{4/3}),
  	  \end{equation}
  	  where we have used that $d_{\max}=o(n^{1/3})$. 
  	  Therefore, we obtain
  	  \begin{equation}
  	  \begin{aligned}[b]
  	  \sum_s s^2n_s^{\sss{\geq T}}
  	  &=
  	  \sum_s s^2n_s\left(1-To(1)\right)+o_\Prob(n^{4/3})\geq \varepsilon n^{4/3}(1-To(1)+o_\Prob(1)).
  	  \end{aligned}
  	  \end{equation}
  	  Taking the limit first for $n\to\infty$, and then for $T\to\infty$ shows that
  	  \begin{equation}
  	  \lim_{T\to\infty}\lim_{n\to\infty}\Prob\left(\sum_{i\geq i_T}v(\mathscr{C}_{(i)})^2>\delta n^{4/3}\right)\neq 0
  	  \end{equation}
  	  for $\delta<\varepsilon$, hence the component sizes do not converge in the $\ltd$-topology.
  	
  	 Now we show that $\Ex{S_n^2}=o(n^{1/3})$ is sufficient for convergence in the $\ltd$-topology.
  	  Let $G^{\sss{\geq T}}$ denote the graph that is obtained by removing all components that have been explored before time $Tn^{2/3}$.
  	 To show that $\Ex{S_n^2}=o(n^{1/3})$ is sufficient for convergence in the $\ltd$ topology, we calculate
  	 \begin{equation}\label{eq:ltdcond}
  	 \begin{aligned}[b]
  	 \Prob\left(\sum_{i\geq i_T}v(\mathscr{C}_{(i)})^2>\delta n^{4/3}\right)&\leq \frac{1}{\delta n^{4/3}}\Ex{\sum_{i\geq i_T}v(\mathscr{C}_{(i)})^2}=\frac{1}{\delta n^{1/3}}\Ex{v(\mathscr{C}^{\sss{\geq T}}(V_n))}\\
  	 &=\frac{1}{\delta n^{1/3}}\Ex{S_{H_n}^{\sss{\geq T}}v(\mathscr{C}^{\sss{\geq T}}(H_n))},
  	 \end{aligned}
  	 \end{equation}
  	 where $V_n$ denotes a randomly chosen vertex of $G^{\sss{\geq T}}$, and $H_n$ denotes a randomly chosen community. Furthermore,  	
  	 \begin{equation}\label{eq:ltd2}
  	 \begin{aligned}[b]
  	 \Ex{S_{H_n}^{\sss{\geq T}}v(\mathscr{C}^{\sss{\geq T}}(H_n))}=\sum_{k,s}p_{k,s}^{\sss{\geq T}}(n)s\Ex{v(\mathscr{C}^{\sss{\geq T}}_{k,s})},
  	 \end{aligned}
  	 \end{equation}
  	 where $v(\mathscr{C}^{\sss{\geq T}}_{k,s})$ denotes the size of a component where the first explored community has size $s$ and inter-community degree $k$.
  	 By~\cite{dhara2016}, the criticality parameter of $G^{\sss{\geq T}}$, $\bar{\nu}_n$, satisfies
  	 \begin{equation}
  	 \label{eq:nubar}
  	 \bar{\nu}_n\leq \nu_n-CTn^{-1/3}+o_\Prob(n^{-1/3}).
  	 \end{equation}
  	  Then, combining Lemma~\ref{lem:Evks} and~\eqref{eq:ltd2} gives
  	  \begin{equation}\label{eq:ESHC}
  	  \begin{aligned}[b]
  	  \Ex{S_{H_n}v(\mathscr{C}^{\sss{\geq T}}(H_n))}=\Ex{(S_n^{\sss{\geq T}})^2}+K\Ex{(D_nS_n)^{\sss{\geq T}}}\frac{n^{1/3}}{CT-\lambda}.
  	  \end{aligned}
  	  \end{equation}
  	 Furthermore, $\Ex{(S_n^{\sss{\geq T}})^2} \leq \Ex{S_n^2}/(n-Tn^{2/3})= \Ex{S_n^2}(1+O(n^{-1/3}))$.
  	  By assumption, $\Ex{S_n^2}=o(n^{1/3})$. Combining this with~\eqref{eq:ltdcond} and~\eqref{eq:ESHC} yields
  	  \begin{equation}
  	  \begin{aligned}[b]
  	  \Prob\left(\sum_{i\geq i_T}v(\mathscr{C}_{(i)})^2>\delta n^{4/3}\right)&\leq o_\Prob(1)+\frac{K}{\delta (CT-\lambda)},
  	  \end{aligned}
  	  \end{equation}
  	  so that
  	  \begin{equation}
  	  \lim_{T\to\infty}\lim_{n\to\infty}\Prob\left(\sum_{i\geq i_T}v(\mathscr{C}_{(i)})^2>\delta n^{4/3}\right)=0.
  	  \end{equation}
  	  Thus, if $\Ex{S_n^2}=o(n^{1/3})$, then Theorem~\ref{thm:crit} also holds in the $\ltd$ topology.
  \end{proof}

  \section{Percolation on the HCM}\label{sec:perc}
In this section we prove Theorem~\ref{thm:prc}, which identifies the scaling limit for the cluster sizes of a HCM under critical percolation.
As described in Section~\ref{sec:percresults}, it is convenient to percolate first only the edges inside communities. This percolation results in a HCM with percolated communities. These percolated communities may be disconnected. However, if we define the connected components of the percolated communities as new communities, we retrieve an updated HCM. After this, we percolate the inter-community connections.
These edges are distributed as in the CM. Therefore, for this second step of percolation, we follow a similar approach as in~\cite{janson2009b}. Combining these two steps of percolation results in the following algorithm that constructs a percolated HCM:
\begin{algorithm}\label{alg:perc}
	\begin{enumerate}
		\item[{(S1)}]
		For each community $H$, remove every edge in $H$ independently with probability $1-\pi$. Let $\bar{n}$ denote the number of connected components of communities after percolation inside the communities. Then, define the connected components of the percolated communities as new communities $(H^\prc_i)_{i\in[\bar{n}]}$.
		\item [{(S2)}]
		Let $H_\mathrm{e}^\prc$ be the percolated community attached to inter-community half-edge $e$. Then, every inter-community half-edge $e$ \textit{explodes} with probability $1-\sqrt{\pi}$, it detaches from $H_\mathrm{e}^\prc$, and is associated to a new community $H'^\prc_e$ of the same shape, but with $e$ as its only inter-community half-edge. Let $n_{H+}$ denote the number of new communities of shape $H$ that are created in this way, and $\tilde{n}=\bar{n}+\sum_H n_{H+}$. Let $(\tilde{H}_i^\prc)_{i\in [\tilde{n}]}$ be the new communities after detaching the half-edges.
		\item [{(S3)}]
		Construct a hierarchical configuration model with community sequence $(\tilde{H}_i^\prc)_{i\in [\tilde{n}]}$.
		\item [{(S4)}]
		For all community shapes $H$, delete the exploded communities with inter-community degree one.
	\end{enumerate}
\end{algorithm}
Figure~\ref{fig:prcalg} illustrates Algorithm~\ref{alg:perc}. By~\cite{janson2009a}, a similar algorithm creates a percolated CM. Therefore, by adding the extra step of percolation inside communities, Algorithm~\ref{alg:perc} creates a percolated HCM.
In this bond percolation procedure, there are three sources of randomness: the percolated communities $H^\prc$ are random, the explosion procedure is random, and then pairing the edges to construct a HCM is random as well.

\begin{figure}[htb]
	\begin{subfigure}[t]{0.31\textwidth}
		\centering
		\includegraphics[width=0.86\linewidth]{comexstubs.pdf}
		\caption{A set of communities} \label{fig:comstubs}
	\end{subfigure}
	\hspace*{\fill} 
	\begin{subfigure}[t]{0.31\textwidth}
		\centering
		\includegraphics[width=0.86\linewidth]{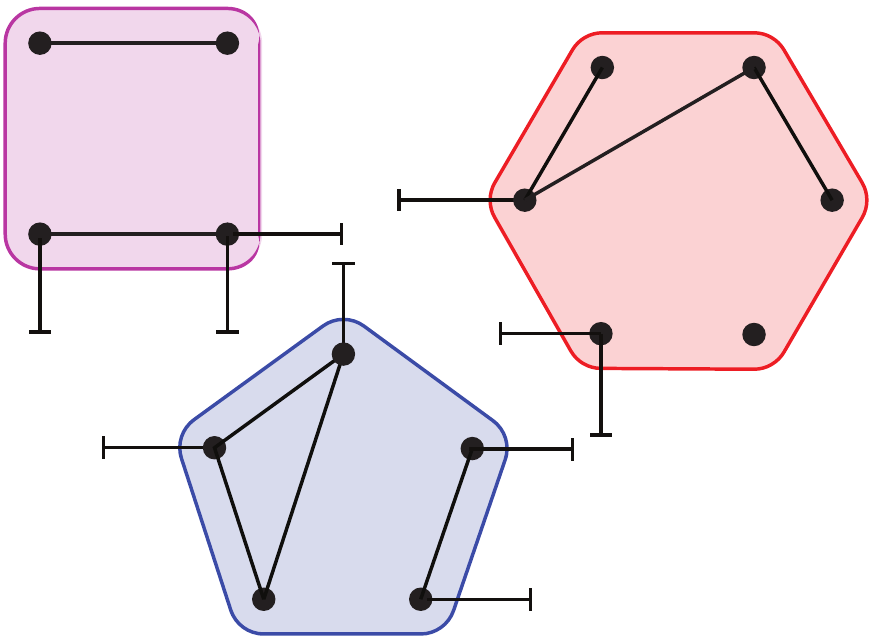}
		\caption{Step (S1): percolation inside communities} \label{fig:percin}
	\end{subfigure}
	\hspace*{\fill} 
	\begin{subfigure}[t]{0.31\textwidth}
		\centering
		\includegraphics[width=0.86\linewidth]{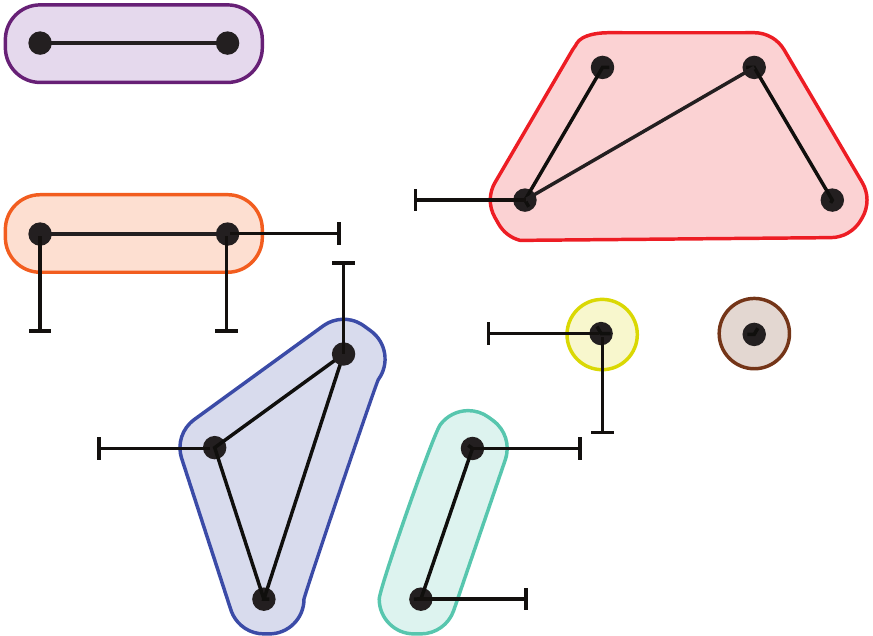}
		\caption{New communities $H^\prc$} \label{fig:percinnew}
	\end{subfigure}\\ [0.8cm]
		\begin{subfigure}[t]{0.31\textwidth}
			\centering
			\includegraphics[width=0.86\linewidth]{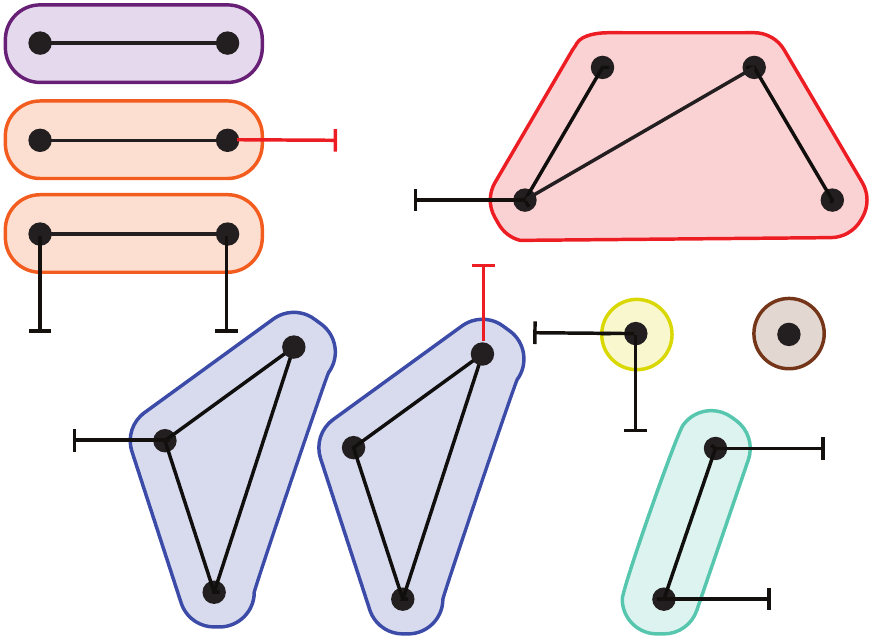}
			\caption{Step {{(S2)}}: exploding half-edges (red) results in communities $\tilde{H}^\prc$} \label{fig:expl}
		\end{subfigure}
		\hspace*{\fill} 
		\begin{subfigure}[t]{0.31\textwidth}
			\centering
			\includegraphics[width=0.86\linewidth]{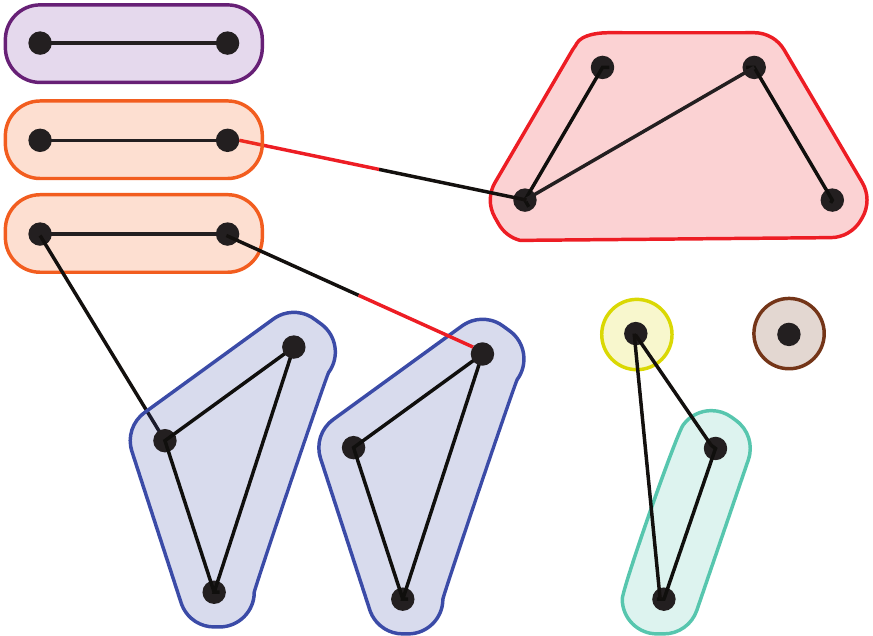}
			\caption{Step {{(S3)}}: connect as in the CM} \label{fig:perccon}
		\end{subfigure}
		\hspace*{\fill} 
		\begin{subfigure}[t]{0.31\textwidth}
			\centering
			\includegraphics[width=0.86\linewidth]{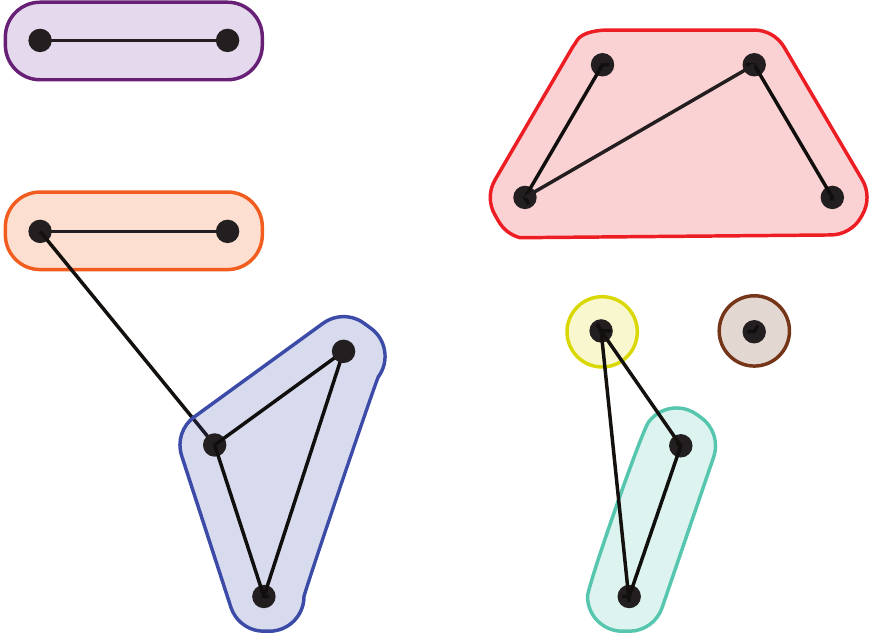}
			\caption{Step {{(S4)}}: delete exploded communities} \label{fig:percout}
		\end{subfigure}
	\caption{Illustration of Algorithm~\ref{alg:perc}. In this example $n=3,\bar{n}=7, \tilde{n}=9$.}
	\label{fig:prcalg}
\end{figure}

\begin{remark}
	In percolation on the regular configuration model,~\cite{janson2009b} showed that instead of deleting the exploded vertices $n_+$, it is also possible to choose $n_+$ vertices uniformly at random from all vertices with degree one, and to delete them. This procedure also results in a multigraph with the same distribution as a percolated configuration model. Similar to this, in our setting it is possible to replace step (S4) of Algorithm~\ref{alg:perc} by
	\begin{enumerate}
		\item [{(S4')}] For all community shapes $H$, choose $n_{H+}$ communities uniformly at random from all communities of shape $H$ and inter-community degree one. Delete these communities.
	\end{enumerate}
\end{remark}

\begin{remark}
	In the CM, each percolated half-edge is replaced by a single half-edge attached to a new vertex. In Algorithm~\ref{alg:perc}, each percolated inter-community half-edge is attached to a new community of the same shape as the original community, but with only one half-edge adjacent to it (see Figure~\ref{fig:expl}). This difference is caused by the fact that in the HCM, communities of the same inter-community degree do not have to be equal. Different communities with inter-community degree $k$ may have different sizes. Therefore, the effect on the component sizes of percolating a half-edge of a community of inter-community degree $k$ is not the same if the community sizes are not the same. Percolating the half-edge adjacent to a larger community has more effect on the component sizes than percolating the half-edges adjacent to a smaller community. For this reason, we replace exploded half-edges by half-edges attached to a community of the same size as the original community, instead of replacing it by a vertex of degree one.
\end{remark}

\subsection{The sizes of critical percolation clusters.}
We now analyze Algorithm~\ref{alg:perc} to prove Theorem~\ref{thm:prc}.
Let $S_n^\prcn$ and $D_n^\prcn$ denote the size and degree of communities after percolation only inside the communities with probability $\pi_n$, and $S^\prc$ and $D^\prc$ their infinite size limits~\cite{hofstad2015}. Furthermore, let $g(H,v,k,\pi_n)$ denote the probability that after percolating community $H$ with parameter $\pi_n$, the connected component containing vertex $v$ contains $k$ half-edges.
By~\cite[(41)]{hofstad2015},
\begin{equation}\label{eq:Dprc}
\Prob(D^\prc=k)=\frac{\sum_H\sum_{v\in V_H}P(H)d_v^{\sss{(b)}}g(H,v,k,\pi)/k}{\sum_H\sum_{v\in V_H}\sum_lP(H)d_v^{\sss{(b)}}g(H,v,l,\pi)/l}.
\end{equation}

We denote the number of communities in the original graph by $n$, the number of communities after percolating only the intra-community edges by $\bar{n}$, and the number of communities after the explosion procedure by $\tilde{n}$. Note that after percolation inside the communities, the number of vertices $N$ remains the same.
Furthermore, similarly to~\cite{dhara2016}, let $\Prob_{\pi_n}^{\bar{n}}$ denote the probability measure containing the shapes of the exploded communities after Algorithm~\ref{alg:perc}, step {(S2)}. Then $\Prob_\pi$ denotes the product measure of $(\Prob_{\pi_n}^{\bar{n}})_{\bar{n}\geq 1}$.
Since $n_{H+}\sim \text{Bin}(d_H\bar{n}_{H},1-\sqrt{\pi_n})$, a.s. with respect to $\Prob_\pi$
\begin{equation}\label{eq:n+s}
n_{H+}=d_H \bar{n}_{H}(1-\sqrt{\pi_n})+o\left(d_H\bar{n}_{H}\right).
\end{equation}
Therefore,
\begin{equation}\label{eq:ntild}
\frac{\tilde{n}}{\bar{n}}=1+\frac{\sum_Hn_{H+}}{\bar{n}}=1+\Ex{D^\prc}(1-\sqrt{\pi})+o(1),
\end{equation}
a.s. with respect to $\Prob_\pi$.

The following lemma proves that the HCM with community sequence $(\tilde{H}_i^\prcn)_{i\in [\tilde{n}]}$ satisfies Conditions~\ref{cond:graph} and~\ref{cond:size}, so that we can apply Theorem~\ref{thm:crit} to find its component sizes:

\begin{lemma}\label{lem:perccond}
	Let $G$ be a hierarchical configuration model satisfying \textup{Condition}~\ref{cond:perc} with community sequence $(H_i)_{i\in[n]}$. Then the hierarchical configuration model with community sequence $(\tilde{H}_i^\prcn)_{i\in [\tilde{n}]}$, constructed as described in \textup{Algorithm}~\ref{alg:perc}, satisfies \textup{Conditions}~\ref{cond:graph} and~\ref{cond:size}.
\end{lemma}
\begin{proof}
	By~\eqref{eq:Dprc},
	\begin{equation}
	\Ex{(D_n^\prcn)^3}=\frac{\sum_H\sum_{v\in V_H}\sum_k P_n(H)d_v^{\sss{(b)}}g(H,v,k,\pi_n)k^2}{\sum_H\sum_{v\in V_H}\sum_l P_n(H)d_v^{\sss{(b)}}g(H,v,k,\pi_n)/l}.
	\end{equation}
	Let $H^\prcn_v$ denote the connected component of the percolated community $H$ containing vertex $v$. Then,
	\begin{equation}\label{eq:D3bound}
	\begin{aligned}[b]
		\sum_{v\in V_H}\sum_k P_n(H)d_v^{\sss{(b)}}g(H,v,k,\pi_n)k^2&=\sum_{v\in V_H} P_n(H)d_v^{\sss{(b)}}\Ex{(\# \text{ outgoing edges of }H_v^\prcn)^2}\\
		&\leq \sum_{v\in V_H} P_n(H)d_v^{\sss{(b)}}d_H^2=P_n(H)d_H^3.
	\end{aligned}
	\end{equation}
	To show that $\Ex{(D_n^\prcn)^3}$ converges, we use the General Lebesgue Dominated Convergence Theorem (see for example~\cite[Thm. 19]{royden2010}), which states that if $\abs{f_n(x)}\leq g_n(x)$ for all $x\in E$,  $\sum_{x\in E} g_n(x)\to \sum_{x\in E} g(x)<\infty$, and $f_n$ converges pointwise to $f$, then also $\sum_{x\in E}f_n(x)\to\sum_{x\in E}f_n(x)$.
By condition~\ref{cond:size}, $\Ex{D_n^3}\to\Ex{D^3}$, so by the General Lebesgue Dominated Convergence Theorem and~\eqref{eq:D3bound}, $\Ex{(D_n^\prcn)^3}\to \Ex{(D^{\prc})^3}$.
Similarly,
\begin{equation}
\Ex{D_n^\prcn S_n^\prcn}=\frac{\sum_H\sum_{v\in V_H}\sum_k P_n(H)d_v^{\sss{(b)}}g(H,v,k,\pi_n)s_{H^\prc_v}}{\sum_H\sum_{v\in V_H}\sum_l P_n(H)d_v^{\sss{(b)}}g(H,v,k,\pi_n)/l}.
\end{equation}
We can bound the summands in the numerator as
\begin{equation}
\begin{aligned}[b]
	\sum_{v\in V_H}\sum_k P_n(H)d_v^{\sss{(b)}}g(H,v,k,\pi_n)s_{H^\prc_v}&\leq \sum_{v\in V_H}\sum_k P_n(H)d_v^{\sss{(b)}}g(H,v,k,\pi_n)s_{H}\\
	&=P_n(H)d_Hs_H,
\end{aligned}
\end{equation}
so that again by the General Lebesgue Dominated Convergence Theorem and Condition~\ref{cond:graph} $\Ex{S^\prcn_n D^\prcn_n}\to\Ex{S^\prc D^\prc}$.
		By a similar reasoning $\Ex{S^\prcn_n}\to\Ex{S^\prc}$.
		Thus, we have proved that $D_n^\prcn$ and $S_n^\prcn$ satisfy Conditions~\ref{cond:graph} and~\ref{cond:size}\ref{cond:ED}. Hence, after percolating inside the communities, the HCM still satisfies these conditions.
		
		We want to prove that $\tilde{D}_n^\prcn$ and $\tilde{S}_n^\prcn$ also satisfy Conditions~\ref{cond:graph} and~\ref{cond:size}\ref{cond:ED}, so that after the explosion process the conditions are still satisfied.
		Since $D_n^\prcn$ satisfies Condition~\ref{cond:size},~\cite[Lemma 24]{dhara2016} shows that $\tilde{D}_n^\prcn$ also satisfies Condition~\ref{cond:size}.
		
		Now we prove the convergence of the first moment of $\tilde{S}_n^\prcn$.
		After explosion, the first $\bar{n}$ entries of $(\tilde{S}^\prcn_i) _{i\in[\tilde{n}]}$ are the same as in $({S}^\prcn_i) _{i\in[\bar{n}]}$, since the community sizes are not changed when percolating the inter-community edges. Furthermore, there are $n_{H+}$ duplicated communities of shape $H$. Thus, the limiting distribution $(\tilde{S}^\prc,\tilde{D}^\prc)$ can be written as
		\begin{equation}
		\begin{aligned}[b]
		\Prob\left(\tilde{S}^\prc = s,\tilde{D}^\prc = k\right) & = \frac{\Prob\left(S^\prc =s,D^\prc = k\right)}{1+\Ex{D^\prc}(1-\sqrt{\pi})}\\
		&\quad +\ind{k=1}\frac{\sum_j j(1-\sqrt{\pi})\Prob\left(S^\prc =s,D^\prc = j\right)}{1+\Ex{D^\prc}(1-\sqrt{\pi})}.
		\end{aligned}
		\end{equation}
				By~\eqref{eq:n+s} and~\eqref{eq:ntild}
		\begin{equation}
		\begin{aligned}[b]
		\frac{1}{\tilde{n}}\sum_{i\in[\tilde{n}]}\tilde{s}_i^\prcn &=\frac{1}{\tilde{n}}\left(\sum_{i\in[\bar{n}]}s_i^\prcn+\sum_{i=\bar{n}+1}^{\tilde{n}}\tilde{s}_i^\prcn\right)
		=\frac{\bar{n}}{\tilde{n}}\Ex{S_n^\prcn}+\frac{\sum_H s_Hn_{H+}}{\tilde{n}}\\
		&=\frac{\Ex{S^\prcn_n}+(1-\sqrt{\pi_n})\Ex{D^\prcn_nS^\prcn_n}}{1+\Ex{D^\prcn_n}(1-\sqrt{\pi_n})}+o(1),
		\end{aligned}
		\end{equation}
		so that $\Ex{\tilde{S}^\prcn_n}\to\Ex{\tilde{S}^{\prc}}$. Furthermore,
		\begin{equation}\label{eq:kstilde}
		\begin{aligned}[b]
		\frac{1}{\tilde{n}}\sum_{i\in[\tilde{n}]} \tilde{s}_i^\prcn \tilde{d}_i^\prcn=\frac{1}{\tilde{n}}\sum_{i\in[\bar{n}]}s_i^\prcn {d}_i^\prcn,
		\end{aligned}
		\end{equation}
		and therefore the combined moment also converges, and Condition~\ref{cond:graph} is satisfied.
		
		To prove Condition~\ref{cond:nu}, note that
		\begin{equation}
		\begin{aligned}[b]
			\nu_{\tilde{D}_n^\prcn}&=\frac{\sum_{i\in[\tilde{n}]}\tilde{d}_i^\prcn(\tilde{d}_i^\prcn -1)}{\sum_{i\in[\tilde{n}]}\tilde{d}_i^\prcn}
			=\frac{\pi_n\sum_{i\in[\bar{n}]}d_i^\prcn(d_i^\prcn -1)+o(n^{2/3})}{\sum_{i\in[\bar{n}]}d_i^\prcn}\\
			&= \pi_n\nu_{D_n^\prcn}+o(n^{-1/3})=1+\lambda n^{-1/3}+o(n^{-1/3}),
		\end{aligned}
		\end{equation}
		where the second equality follows from~\cite[equation (7.2)]{dhara2016}.
\end{proof}

\begin{remark}
	In Lemma~\ref{lem:perccond}, we have assumed that the HCM satisfies Conditions~\ref{cond:graph} and~\ref{cond:size}~\ref{cond:ED} and~\ref{cond:PD} before percolation. Then, we have shown that $S_n^\prc$ and $D_n^\prc$ also satisfy these conditions. However, it is also possible to assume from the start that $S_n^\prc$ and $D_n^\prc$ satisfy these conditions. This means for example that the inter-community degrees only need to have finite third moment after percolating inside the communities, they may have an infinite third moment before percolating inside the communities. We gave an example of such a community in Remark~\ref{rem:perc}. 
\end{remark}

\begin{proof}[Proof of Theorem~\ref{thm:prc}]
After explosion, the HCM satisfies the assumptions of Theorem~\ref{thm:crit} by Lemma~\ref{lem:perccond}. Then Theorem~\ref{thm:crit} gives the component sizes of the exploded HCM. To obtain the sizes of the components of the percolated HCM, we need to know how many vertices are deleted in the last step of Algorithm~\ref{alg:perc}. We denote the number of deleted vertices from component $\mathscr{C}_j$ by $v^d(\mathscr{C}_j)$.
 If a community of size $s$ is deleted, $s$ vertices are deleted. Thus, this number can be written as
 \begin{equation}
 v^d(\tilde{\mathscr{C}}_{(j)})=\sum_{i=1}^{v^\hie(\tilde{\mathscr{C}}_{(j)})}s_{H_i}\mathbbm{1}_{\{H_i\text{ is deleted}\}}.
 \end{equation}
 Let $\tilde{n}_{H,1}$ denote the number of communities of shape $H$ and inter-community degree one.
 Using~\cite[Proposition 29]{dhara2016} with $\alpha=2/3$ and $f_n(i)$ as the indicator function that community $i$ is of shape $H$ and has inter-community degree one, we can show that the number of communities of shape $H$ with inter-community degree one in a component $\mathscr{C}_{(j)}$ satisfies
 \begin{equation}
 v_{H,1}^\hie(\mathscr{C}_{(j)})=v^\hie(\mathscr{C}_{(j)})\frac{\tilde{n}_{H,1}}{\tilde{n}\Ex{\tilde{D}^\prcn}}+o_{\Prob_\pi}\left(n^{1/3}\frac{\tilde{n}_{H,1}}{\sum_H\tilde{n}_{H,1}}\right).
 \end{equation}
Therefore, the number of vertices in communities of shape $H$ with inter-community degree one in a component $\mathscr{C}_{(j)}$ satisfies
\begin{equation}
v_{H,1}(\mathscr{C}_{(j)})=v^\hie(\mathscr{C}_{(j)})\frac{s_H\tilde{n}_{H,1}}{\tilde{n}\Ex{\tilde{D}^\prcn}}+o_{\Prob_\pi}\left(n^{1/3}\frac{s_H\tilde{n}_{H,1}}{\sum_H\tilde{n}_{H,1}}\right).
\end{equation}
A fraction of $n_{H+}/\tilde{n}_{H,1}$ communities of shape $H$ with outside degree one is removed uniformly. Therefore, for $j$ fixed,
 \begin{equation}\label{eq:+1}
\begin{aligned}
 v^d(\tilde{\mathscr{C}}_{(j)})&=v^\hie(\tilde{\mathscr{C}}_{(j)})\frac{\sum_H s_Hn_{H+}}{\tilde{n}\Ex{\tilde{D}^\prcn}}+o_{\Prob_\pi}\left(n^{1/3}\frac{\sum_Hs_H\tilde{n}_{+H}}{\sum_H\tilde{n}_{H,1}}\right)+o_{\Prob_\pi}\left(\sum_H {n}_{+H}\right)\\
 &=v^\hie(\tilde{\mathscr{C}}_{(j)})\frac{\sum_H s_Hn_{H+}}{\tilde{n}\Ex{\tilde{D}^\prcn}}+o_{\Prob_\pi}(n^{1/3})+o_{\Prob_\pi}(n^{2/3}),
\end{aligned}
 \end{equation}
 where we used~\eqref{eq:n+s} and the fact that $\tilde{n}_{H,1}\geq n_{H+}$.
Thus, by~\eqref{eq:+1},~\eqref{eq:n+s},~\eqref{eq:kstilde} and~\eqref{eq:critsize},
\begin{equation}
\begin{aligned}[b]
v^d(\tilde{\mathscr{C}}_{(j)})&=v^\hie(\tilde{\mathscr{C}}_{(j)})\frac{\sum_H s_H{n}_{+H}}{\tilde{n}\Ex{\tilde{D}^\prcn}}+o_{\Prob_\pi}(n^{2/3})\\
&= \frac{(1-\sqrt{\pi_n})\sum_{H}d_Hs_H \bar{n}_H}{\tilde{n}\Ex{\tilde{D}^\prcn}}v^\hie(\tilde{\mathscr{C}}_{(j)})+o_{\Prob_\pi}(n^{2/3})\\
&= (1-\sqrt{\pi_n})\frac{\sum_{k,s} k s \bar{n}_{k,s}}{\tilde{n}\Ex{\tilde{D}^\prcn}}v^\hie(\tilde{\mathscr{C}}_{(j)})+o_{\Prob_\pi}(n^{2/3})\\
&=(1-\sqrt{\pi_n})\frac{\sum_{k,s}ks\tilde{n}_{k,s}}{\tilde{n}\Ex{\tilde{D}^\prcn}}v^\hie(\tilde{\mathscr{C}}_{(j)})+o_{\Prob_\pi}(n^{2/3})\\
&=(1-\sqrt{\pi_n})\frac{\Ex{\tilde{D}^\prcn \tilde{S}^\prcn}}{\Ex{\tilde{D}^\prcn}}v^\hie(\tilde{\mathscr{C}}_{(j)})+o_{\Prob_\pi}(n^{2/3})\\
&=(1-\sqrt{\pi_n})v(\tilde{\mathscr{C}}_{(j)})+o_{\Prob_\pi}(n^{2/3}).
\end{aligned}
\end{equation}
Then, Theorem~\ref{thm:crit} gives
\begin{equation}
\tilde{n}^{-2/3}(v(\mathscr{C}_{(j)}))_{j\geq 1}\indist \frac{\Ex{\tilde{D}^\prc\tilde{S}^\prc}}{\Ex{\tilde{S}^\prc}}\sqrt{\pi}\tilde{\boldsymbol{\gamma}}^\lambda.
\end{equation}
Noting that $N=\sum_{i=1}^{\tilde{n}}\tilde{s}^\prcn_i$ leads to
\begin{equation}
\frac{N}{\bar{n}}\inprob \Ex{\tilde{S}^\prc},
\end{equation}
so that~\eqref{eq:sizeperc} follows.
\end{proof}

\subsection{The critical window.}\label{sec:exppi}
Equation~\eqref{eq:pic} gives an implicit equation for the critical window. We want to know whether it is possible to write~\eqref{eq:pic} in the form
\begin{equation}\label{eq:pila}
\pi_n(\lambda)=\pi_n(0)\left(1+\frac{\lambda c^*}{n^{1/3}}\right)+o(n^{-1/3}),
\end{equation}
for some $c^*\in\mathbb{R}$, so that the width of the critical window in the hierarchical configuration model is similar to the width of the critical window in the configuration model.

Since $g(H,v,k,\pi_n(\lambda))$ is not necessarily increasing in $\lambda$, we rewrite~\eqref{eq:pic} as
\begin{equation}\label{eq:picin}
\pi_n(\lambda)=\frac{\mathbb{E}[D_n]}{\sum_{H}P_n(H)\sum_{v\in V_H}d_v^{\sss{(b)}}\sum_{k=1}^{D_H-1}B(H,v,k+1,\pi_n(\lambda))}\left(1+\frac{\lambda}{n^{1/3}}\right),
\end{equation}
where $B(H,v,k,\pi_n(\lambda))$ is the probability that after percolating $H$ with probability $\pi_n(\lambda)$, the connected component of $v$ contains at least $k$ inter-community half-edges, which is increasing in $\lambda$.

\begin{lemma}\label{lem:exppic}
	For a hierarchical configuration model satisfying \textup{Condition}~\ref{cond:perc} and $\lim_{n\to\infty}\Ex{D_n^2S_n}=\Ex{D^2S}<\infty$,
	\begin{equation}\label{eq:critex}
	\pi_n(\lambda)=\pi_n(0)\left(1+\frac{\lambda c^*}{n^{1/3}}\right)+o(\lambda n^{-1/3}),
	\end{equation}
	where
	\begin{equation}\label{eq:lstar}
	c^*= \frac{\Ex{D}}{\Ex{D}+\pi^2\sum_HP(H)\sum_v d_v^{\sss{(b)}}\sum_k \frac{d}{dp}B(H,v,k+1,p)_{p=\pi}}.
	\end{equation}
\end{lemma}
\begin{remark}
Equation~\eqref{eq:lstar} shows that $c^*\leq 1$, so that the critical window of a HCM is smaller than the critical window of a CM where no communities are inserted. Here $\frac{d}{dp}B(H,v,k,p)_{p=\pi}$ captures how vulnerable community $H$ is to percolation inside the community. The larger $\frac{d}{dp}B(H,v,k,p)_{p=\pi}$ will be, the larger the difference between $\lambda$ and $\lambda c^*$ will be. Intuitively, when $\frac{d}{dp}B(H,v,k,p)_{p=\pi}$ is small, this indicates that changing the percolation probability changes the degrees of the percolated communities very little. Therefore, the critical behavior is almost entirely explained by the macroscopic CM in that case. On the other hand, when $\frac{d}{dp}B(H,v,k,p)_{p=\pi}$ is large, increasing the percolation probability by a small amount increases the degrees of the communities by a lot. Then $\lambda c^*$ may be much smaller than $\lambda$.
\end{remark}

\begin{proof}
We can write~\eqref{eq:picin} as
\begin{equation}
\pi_n(\lambda)=L_n(\pi_n(\lambda))\left(1+\frac{\lambda}{n^{1/3}}\right),
\end{equation}
where
\begin{equation}
L_n(\pi_n(\lambda))=\frac{\mathbb{E}[D_n]}{\sum_{H}P_n(H)\sum_{v\in V_H}d_v^{\sss{(b)}}\sum_{k=1}^{D_H-1}B(H,v,k+1,\pi_n(\lambda))}.
\end{equation}
Calculating the derivative gives
\begin{equation}
\pi_n'(\lambda)=\frac{L_n(\pi_n(\lambda))}{n^{1/3}(1-L_n'(\pi_n(\lambda))(1+\lambda/n^{1/3}))}
\end{equation}
Then, by the mean value theorem, there exists $\lambda^*\in[0,\lambda]$ such that
\begin{equation}\label{eq:mv}
\begin{aligned}[b]
\pi_n(\lambda)
&=\pi_n(0)+\frac{\lambda}{n^{1/3}}\frac{L_n(\pi_n(\lambda^*))}{1-L_n'(\pi_n(\lambda^*))(1+\lambda^*/n^{1/3})}.
\end{aligned}
\end{equation}
Since $B(H,v,k,\pi)$ is the probability of an increasing event, $L_n(\pi_n(\lambda))$ is continuous. Calculating the derivative of $L_n(\pi_n(\lambda))$ gives
\begin{equation}
L_n'(\pi_n(\lambda))=-\frac{\mathbb{E}[D_n]\sum_{H}P_n(H)\sum_vd_v^{\sss{(b)}}\sum_kB'(H,v,k+1,\pi_n(\lambda))}{\left(\sum_HP_n(H)\sum_vd_v^{\sss{(b)}}\sum_kB(H,v,k+1,\pi_n(\lambda))\right)^2},
\end{equation}
where we denoted
\begin{equation}
B'(H,v,k,\pi) = \frac{d}{dp}B(H,v,k,p)_{p=\pi}.
\end{equation}
Since $B$ is an increasing function of the percolation parameter $p$, $L_n'(\pi_n(\lambda))\leq 0$.
 By~\cite[Thm. 9]{hofstad2015}
\begin{equation}\label{eq:nun}
\nu_{D^\prcn}^{\sss{(n)}}=\frac{\sum_{H}P_n(H)\sum_{v\in V_H}d_v^{\sss{(b)}}\sum_{k=1}^{D_H-1}B(H,v,k+1,\pi_n)}{\mathbb{E}[D_n]},
\end{equation}
Therefore, by~\eqref{eq:condin} and~\eqref{eq:nun},
\begin{equation}
\begin{aligned}[b]
\abs{L'(\pi_n(\lambda))}&=\frac{\mathbb{E}[D_n]\sum_{H}P_n(H)\sum_vd_v^{\sss{(b)}}\sum_kB'(H,v,k+1,\pi_n(\lambda))}{\left(\sum_HP_n(H)\sum_vd_v^{\sss{(b)}}\sum_kB(H,v,k+1,\pi_n(\lambda))\right)^2}\\
&\leq \frac{\sum_{H}P_n(H)\sum_vd_v^{\sss{(b)}}\sum_kB'(H,v,k+1,\pi_n(\lambda))}{\mathbb{E}[D_n]}.
\end{aligned}
\end{equation}
Hence, we need to bound $B'(H,v,k,\pi_n(\lambda))$. The event that $v$ is connected to at least $k$ half-edges is increasing in $\lambda$. Let $\mathcal{E}(H,v,k)$ denote the event that vertex $v$ is connected to at least $k$ half-edges of community $H$. An edge is pivotal for $\mathcal{E}(H,v,k)$ in a certain configuration, if the event occurs if the edge is present, and the event does not occur if the edge is not present. Then, by Russo's formula~\cite{russo1981},
\begin{equation}
\begin{aligned}[b]
B'(H,v,k,\pi_n(\lambda))&=\sum_{e\in H}\Prob_{\pi_n(\lambda)}(e\text{ pivotal for }\mathcal{E}(H,v,k+1))\\
&=\frac{1}{\pi_n(\lambda)}\sum_{e\in H}\Prob_{\pi_n(\lambda)}(e\text{ present and pivotal for }\mathcal{E}(H,v,k+1))\\
&=\frac{1}{\pi_n(\lambda)}\mathbb{E}_{\pi_n(\lambda)}[\text{\# pivotal, present edges for }\mathcal{E}(H,v,k+1)]\\
&\leq\frac{1}{\pi_n(\lambda)}(S_H-1),
\end{aligned}
\end{equation}
because at most $S_H-1$ pivotal edges can be present in a community, since otherwise, they would form a cycle. Therefore,
\begin{equation}\label{eq:boundgp}
\begin{aligned}[b]
\abs{L_n'(\pi_n(\lambda))}\leq \frac{\sum_{H}P_n(H)\sum_vd_v^{\sss{(b)}}\sum_{k=1}^{d_H-1}(s_H-1)}{\mathbb{E}[D_n]\pi_n(\lambda)}\leq \frac{\sum_{H}P_n(H)d_H^2s_H}{\mathbb{E}[D_n]\pi_n(\lambda)}.
\end{aligned}
\end{equation}

Since $\Ex{D_n^2S_n}\to \Ex{D^2S}$, we can use the General Lebesgue Dominated Convergence Theorem to conclude that
\begin{equation}
\begin{aligned}[b]
\lim_{n\to\infty} L_n'(\pi_n(\lambda^*))&=-\frac{\Ex{D}\sum_H P(H)\sum_vd_v^{\sss{(b)}}\sum_k B'(H,v,k+1,\pi)}{\left(\sum_HP(H)\sum_vd_v^{\sss{(b)}}\sum_kB(H,v,k+1,\pi)\right)^2}\\
&=-\frac{\pi^2\sum_H P(H)\sum_vd_v^{\sss{(b)}}\sum_k B'(H,v,k+1,\pi)}{\Ex{D}},
\end{aligned}
\end{equation}
where the last inequality follows from~\eqref{eq:nun} and~\eqref{eq:pic}.
Furthermore, we can use the General Lebesgue Dominated Convergence Theorem (see the proof of~\cite[Thm. 9]{hofstad2015}) to conclude that
\begin{equation}
\lim_{n\to\infty}L_n(\pi_n(\lambda))=L(\pi)=\pi.
\end{equation}
Inserting this into~\eqref{eq:mv} proves~\eqref{eq:critex}.
\end{proof}

\begin{example}[Star-shaped communities]
We now consider the case where all communities are star-shaped, so that every community has one vertex in the middle, connected to $l$ other vertices that all have inter-community degree one (as in Figure~\ref{fig:star}). Then~\eqref{eq:pic} becomes
\begin{equation}
\pi_n(\lambda)=\frac{1}{(l-1)\pi_n(\lambda)^2}\left(1+\frac{\lambda}{n^{1/3}}\right),
\end{equation}
or
\begin{equation}
\pi_n(\lambda)=\frac{1}{(l-1)^{1/3}}\left(1+\frac{\lambda}{n^{1/3}}\right)^{1/3}.
\end{equation}
A first order Taylor approximation then gives
\begin{equation}
\begin{aligned}[b]
\pi_n(\lambda)&=\frac{1}{(l-1)^{1/3}}+\frac{\lambda}{3n^{1/3}(l-1)^{1/3}}+O(n^{-2/3})\\
&=\frac{1}{(l-1)^{1/3}}\left(1+\frac{\lambda}{3n^{1/3}}\right)+O(n^{-2/3}),
\end{aligned}
\end{equation}
so that $c^*=1/3$, which is the same result that is obtained when computing~\eqref{eq:critex}.
Table~\ref{tab:piappstar} compares the approximation of~\eqref{eq:critex} with the exact values of $\pi_n(\lambda)$.
\begin{table}[htbp]
		\centering
		\begin{tabular}{rrrrr}
			\toprule
			& \multicolumn{2}{c}{$n=10^5$} & \multicolumn{2}{c}{$n=10^6$} \\
			$\lambda$ & $\pi_n(\lambda)$ & $\pi_n(\lambda)$ appr. & $\pi_n(\lambda)$ & $\pi_n(\lambda)$ appr.\\
			\midrule
		-10   & 0,581 & 0,585 & 0,608 & 0,609 \\
		-1    & 0,625 & 0,625 & 0,628 & 0,628 \\
		0 & 0,630 & 0,630 & 0,630 & 0,630\\
		1     & 0,634 & 0,634 & 0,632 & 0,632 \\
		10    & 0,672 & 0,675 & 0,650 & 0,651 \\
		\bottomrule
	\end{tabular}%
	\caption{Values of $\pi_n(\lambda)$ for star-shaped communities with 5 end points, and the approximation by~\eqref{eq:critex}.}
	\label{tab:piappstar}%
\end{table}%
\end{example}

\begin{figure}[tb]
	\centering
	\begin{minipage}[b]{0.31\linewidth}
		\centering
		\includegraphics[width=0.5\textwidth]{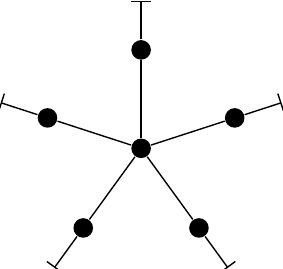}
		\caption{A star-shaped community}
		\label{fig:star}
	\end{minipage}
	\hspace{0.2cm}
	\begin{minipage}[b]{0.31\linewidth}
		\centering
		\includegraphics[width=0.8\textwidth]{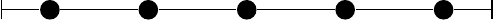}
		\caption{A line community of length 5}
		\label{fig:line}
	\end{minipage}
\end{figure}

\begin{example}[Line communities]
	We now consider the case where all communities are either line communities of length 5 (as in Figure~\ref{fig:line}), or single vertices of degree 3, both with probability 1/2. Here a line community is a community that consists of a line of 5 vertices. The two vertices at the ends of the line have inter-community degree one, the other vertices have inter-community degree zero.
	It is possible to calculate~\eqref{eq:critex} analytically in this setting. Table~\ref{tab:piappr} compares this approximation with the exact values of $\pi_n(\lambda)$. We can see that the approximation is very close to the actual value of $\pi_n(\lambda)$, especially for $n$ large and $\lambda$ small. 
	\begin{table}[htbp]
		\centering
		\begin{tabular}{rrrrr}
			\toprule
			& \multicolumn{2}{c}{$n=10^5$} & \multicolumn{2}{c}{$n=10^6$} \\
			$\lambda$ & $\pi_n(\lambda)$ & $\pi_n(\lambda)$ appr. & $\pi_n(\lambda)$ & $\pi_n(\lambda)$ appr.\\
						\midrule
			-10   & 0,623 & 0,636 & 0,696 & 0,698 \\
			-1    & 0,741 & 0,741 & 0,747 & 0,747 \\
			0 & 0,753 &0,753 &0,753 &0,753\\
			1     & 0,764 & 0,764 & 0,758 & 0,758 \\
			10    & 0,858 & 0,870 & 0,804 & 0,807 \\
			\bottomrule
		\end{tabular}%
				\caption{Values of $\pi_n(\lambda)$ for line communities and single vertex communities, and the approximation by~\eqref{eq:critex}.}
		\label{tab:piappr}%
	\end{table}%
\end{example}

\section{Conclusion}
We investigated the influence of mesoscopic structures on critical component sizes by studying the hierarchical configuration model (HCM). In the HCM, the mesoscopic structures refer to community structures, and the connections between different communities are as in the configuration model. We considered the critical component sizes of the HCM when the inter-community connections have a finite third moment. These critical component sizes converge as $n\to\infty$ to a similar scaling limit as the critical component sizes in the CM, as long as the mesoscopic scales remain smaller than $n^{2/3}$. The critical component sizes of the HCM only depend on the sizes of the communities, and are independent of the precise community shapes. We also obtained an implicit critical percolation window for the HCM, that depends on both the connections between communities, as well as the connections inside communities. We found that under stricter conditions on the community sizes and the inter-community edges, the critical window can be written in an explicit form. The question whether this stricter condition is necessary to write the critical window in an explicit form remains open for further research. 

The HCM can be used to model real-world networks with a community structure. Since many real-world networks have diverging third moments of their inter-community connections~\cite{stegehuis2015}, it would be worthwhile to investigate the scaling limits of the HCM in this setting.

\paragraph*{Acknowledgement.}
This work is supported by NWO TOP grant 613.001.451 and by the NWO Gravitation Networks grant 024.002.003.
The work of RvdH is further supported by the NWO VICI grant 639.033.806.  The work of JvL is further supported by an NWO TOP-GO grant and by an ERC Starting Grant.

\bibliographystyle{abbrv}
\bibliography{../references}
\end{document}